\documentclass[10pt]{amsart}
\usepackage{amssymb,graphicx}
\usepackage{mathrsfs, bbm, slashed}
\usepackage{color}
\usepackage{enumerate}
\usepackage{a4wide}
\usepackage[colorlinks=true,citecolor=blue,linkcolor=blue]{hyperref}

\newcommand{\R}{\mathbb{R}}

\newcommand{\C}{\mathbb{C}}

\newcommand{\Z}{\mathbb{Z}}

\newcommand{\bS}{\mathbb{S}}

\newcommand{\sA}{\mathscr{A}}

\newcommand{\sE}{\mathscr{E}}
\newcommand{\sF}{\mathscr{F}}
\newcommand{\sL}{\mathscr{L}}
\newcommand{\csL}{\overline{\mathscr{L}}}
\newcommand{\sH}{\mathscr{H}}
\newcommand{\sK}{\mathscr{K}}
\newcommand{\sM}{\mathscr{M}}
\newcommand{\sMs}{\mathscr{M}_{\textup{s}}}
\newcommand{\sR}{\mathscr{R}}

\newcommand{\sT}{\mathscr{T }}

\newcommand{\sW}{\mathscr{W}}

 \newcommand{\psido}{$\Psi$DO}
 \newcommand{\psidos}{$\Psi$DOs}

\def\Xint#1{\mathchoice
{\XXint\displaystyle\textstyle{#1}}%
{\XXint\textstyle\scriptstyle{#1}}%
{\XXint\scriptstyle\scriptscriptstyle{#1}}%
{\XXint\scriptscriptstyle\scriptscriptstyle{#1}}%
\!\int}
\def\XXint#1#2#3{{\setbox0=\hbox{$#1{#2#3}{\int}$}
\vcenter{\hbox{$#2#3$}}\kern-.5\wd0}}

\def\dashint{\Xint-}

\newcommand{\bint}{\ensuremath{\dashint}}
\newcommand{\Trwmed}{\Tr_{\omega_{\textup{med}}}}
\newcommand{\bTrwmed}{\overline{\op{Tr}}_{\omega_{\textup{med}}}}
\newcommand{\bTrw}{\overline{\op{Tr}}_{\omega}}
\newcommand{\op}{\operatorname} 

\newcommand{\Sp}{\op{Sp}}
\newcommand{\tr}{\op{tr}}

\newcommand{\Tr}{\op{Tr}}
\newcommand{\Trw}{\Tr_{\omega}}

\newcommand{\Com}{\op{Com}}
\newcommand{\End}{\op{End}}

\newcommand{\Res}{\op{Res}}

\newcommand{\limw}{{\lim}_{\omega}\,} 
\newcommand{\limmed}{\lim\op{med}} 
\newcommand{\med}{\textup{med}}

\newcommand{\scal}[2]{\ensuremath{\left\langle #1 | #2 \right\rangle}} 
\newcommand{\acou}[2]{\ensuremath{\left\langle #1 , #2 \right\rangle}} 
\newcommand{\bigscal}[2]{\ensuremath{\big\langle #1 | #2 \big\rangle}} 
\newcommand{\bigacou}[2]{\ensuremath{\big\langle #1 , #2 \big\rangle}}

\newcommand{\dom}{\op{dom}}

\def\Xint#1{\mathchoice
{\XXint\displaystyle\textstyle{#1}}%
{\XXint\textstyle\scriptstyle{#1}}%
{\XXint\scriptstyle\scriptscriptstyle{#1}}%
{\XXint\scriptscriptstyle\scriptscriptstyle{#1}}%
\!\int}
\def\XXint#1#2#3{{\setbox0=\hbox{$#1{#2#3}{\int}$ }
\vcenter{\hbox{$#2#3$ }}\kern-.6\wd0}}

\numberwithin{equation}{section}

\newtheorem{theorem}{Theorem}[section]
\newtheorem{proposition}[theorem]{Proposition}
\newtheorem{corollary}[theorem]{Corollary}
\newtheorem*{questionA}{Question A}
\newtheorem*{questionB}{Question B}
\newtheorem*{questionC}{Question C}
\newtheorem*{questionD}{Question D}
\newtheorem*{questionE}{Question E}

\newtheorem{lemma}[theorem]{Lemma}

\newtheorem*{conjecture*}{Conjecture}

\theoremstyle{definition}
\newtheorem{definition}[theorem]{Definition}

\theoremstyle{remark}

\newtheorem{remark}[theorem]{Remark}
 
\newtheorem*{claim*}{Claim} 

\newcommand{\LlogL}{L\!\log\!L}

\newcommand{\ssD}{\slashed{D}}

\title{Connes' Integration and Weyl's Laws}
\date{\today}

\author{Rapha\"el Ponge}
 \address{School of Mathematics, Sichuan University, Chengdu, China}
 \email{ponge.math@icloud.com}

\begin{document}
\begin{abstract}
This paper deal with some questions regarding the notion of integral in the framework of Connes's noncommutative geometry. First, we present a purely spectral theoretic construction of Connes' integral. This answers a question of Alain Connes. We also deal with the compatibility of Dixmier traces with Lebesgue's integral. This answers another question of Alain Connes. We further clarify the relationship of Connes' integration with Weyl's laws for compact operators and Birman-Solomyak's perturbation theory. We also give a "soft proof" of Birman-Solomyak's Weyl's law for negative order pseudodifferential operators on closed manifold. This Weyl's law yields a stronger form of Connes' trace theorem. Finally, we explain the relationship between Connes' integral and semiclassical Weyl's law for Schr\"odinger operators. This is an easy consequence of the Birman-Schwinger principle. We thus get a neat link between noncommutative geometry and semiclassical analysis. 
\end{abstract}
\maketitle 

\section{Introduction}
The quantized calculus of Connes~\cite{Co:NCG} aims at translating the main tools of the classical infinitesimal calculus into the operator theoretic language of quantum mechanics. As an Ansatz the integral in this setup should be a positive trace on the weak trace  class $\sL_{1,\infty}$ (see Section~\ref{sec:NCInt}). Natural choices are given by the traces  $\Tr_\omega$ of Dixmier~\cite{Di:CRAS66} (see also~\cite{Co:NCG, LSZ:Book} and Section~\ref{sec:NCInt}). These traces are associated with extended limits. 
 Following Connes~\cite{Co:NCG} we say that an operator $A\in\sL_{1,\infty}$ is measurable when the value of $\Tr_\omega(A)$ is independent of the extended limit.  We then define the NC integral $\bint A$ to be this value. It follows from this construction that if $A\in \sL_{1,\infty}$ is \emph{positive}, then 
\begin{equation*}
 \bigg(\text{$A$ is measurable and}\ \bint A= L\bigg) \Longleftrightarrow \lim_{N\rightarrow \infty} \frac{1}{\log N} \sum_{j<N} \lambda_j(A)=L, 
\end{equation*}
where $\lambda_0(A)\geq \lambda_1(A) \geq \cdots $ are the eigenvalues of $A$ counted with multiplicity. 

During the conference ``Noncommutative geometry: state of the arts and future prospects", which was held at Fudan University in Shanghai, China, from March 29-April 4, 2017,  Alain Connes asked the following question: 

\begin{questionA}[Connes]
 Is it possible to show the existence of a limit for all measurable operators without using extended limits? 
\end{questionA}

In other words, Connes is stressing the need for a purely spectral theoretic construction of the integral in noncommutative geometry. A partial answer to this question was given in a recent preprint of Sukochev-Zanin~\cite{SZ:arXiv21}. However, that paper deals with a special class of operators and the approach still relies on using extended limits at some intermediate step. Therefore, this does not provide a satisfactory answer to Connes' question. 

In this paper we observe that we can answer Connes' question by using a lemma in the 2012 book of Lord-Sukochev-Zanin~\cite{LSZ:Book}. 
 Although the main focus of this book is on singular traces, the authors establish there an interesting asymptotic additivity result for sums of eigenvalues of weak trace class operators. Recall if $A$ is a compact operator,  its spectrum can be organized as a sequence of eigenvalues $(\lambda_j(A))_{j\geq 0}$ such that $|\lambda_0(A)|\geq |\lambda_1(A)|\geq \cdots$, where each eigenvalue is repeated according to its (algebraic) multiplicity. By~\cite[Lemma~5.7.5]{LSZ:Book} if $A$ and $B$ are operators in $\sL_{1,\infty}$, then
 \begin{equation}
 \sum_{j\leq N} \lambda_j(A+B) =  \sum_{j\leq N} \lambda_j(A)  +  \sum_{j\leq N} \lambda_j(B) +\op{O}(1).
 \label{eq:Intro.additivity}  
\end{equation}
This result is related to an eigenvalue characterization of the commutator space $\Com(\sL_{1,\infty})$. It is also used in~\cite{LSZ:Book} to show that an operator $A\in \sL_{1,\infty}$ is measurable if and only it is Tauberian, in the sense that 
\begin{equation}
 \lim_{N\rightarrow \infty} \frac{1}{\log N} \sum_{j<N} \lambda_j(A)\ \text{exists}. 
 \label{eq:Intro.Tauberian-limit}  
\end{equation}

We take the point of view to start from scratch and work with Tauberian operators from the very beginning. We define on such operators a functional $\bint'$ given by the limit in~(\ref{eq:Intro.Tauberian-limit}). It easily follows from~(\ref{eq:Intro.additivity}) that Tauberian operators form a subspace of $\sL_{1,\infty}$ on which $\bint'$ is a positive linear trace, i.e., it satisfies the NC integral's Ansatz  (see Proposition~\ref{prop:NCInt.properties-Lambda}). The key  result is the equality between this functional and the NC integral as defined above (Theorem~\ref{thm:NCTint.Tauberian-measurable}). In particular, if $A$ is measurable, then 
\begin{equation}
 \bint A = \lim_{N\rightarrow \infty} \frac{1}{\log N} \sum_{j<N} \lambda_j(A). 
 \label{eq:Intro.bint-limit} 
\end{equation}
This formula is not stated in~\cite{LSZ:Book}. This gives a purely spectral theoretic construction of Connes' integral, and hence this answers Connes' Question.

Another interesting consequence of the above result is the spectral invariance of Connes' integral. Namely, if two weak trace operators $A$ and $B$ (possibly acting on different Hilbert spaces) have the same non-zero eigenvalues with same multiplicities, then one is measurable if and only if the other is, and this case their NC integrals agree (see Proposition~\ref{prop:NCInt.spectral-invariance}).   

Another important question regarding Connes' integral is its compatibility with Lebesgue's integral. This is a sensible question since $\sL_{1,\infty}$ is a quasi-Banach ideal, but this is not a Banach space or even a locally convex topological vector space. Thus, Bochner intergration and Gel'fand-Pettis integration of maps  with values in $\sL_{1,\infty}$ do not make sense. 

\begin{questionB}[Connes~\cite{Co:Survey19}] 
Can we single out a Dixmier trace that commutes with Lebesgue's integral?
 \end{questionB}
 
We stress that we're seeking for a trace that is defined on all weak trace operators. Connes~\cite{Co:Survey19} actually suggested to use Dixmier traces associated with medial limits in the sense of Mokodoski~\cite{Me:SPS73}. They are extended limits with the fundamental property to be universally measurable and to commute with Lebesgue's integration (see~\cite{Me:SPS73}). 

We observe that any Dixmier trace $\Trw$ uniquely extends to a linear trace $\bTrw:\csL_{1,\infty}\rightarrow \C$, where
 $\csL_{1,\infty}$ is the closure of $\sL_{1,\infty}$ in the Dixmier-Macaev ideal (see Lemma~\ref{lem:Int.extension-tau}).  The advantage of passing to $\csL_{1,\infty}$  is to work with a Banach space, since the Dixmier-Macaev ideal is a Banach ideal. Thus, Bochner integration with values in $\csL_{1,\infty}$  makes sense and commutes with continuous linear forms. We then get for almost free the following compatibility result (Proposition~\ref{prop:Int.compatibility}): if $(\Omega, \mu)$ is a measure space and $A:\Omega\rightarrow \sL_{1,\infty}$ is a measurable map which is integrable as an $\csL_{1,\infty}$-map, then we have 
\begin{equation*}
 \int_\Omega \Trw\big[A(x)\big]d\mu(x)= \bTrw \bigg( \int_\Omega A(x)d\mu(x)\bigg). 
\end{equation*}
In particular, if $(\Omega,\mu)$ is a finite measure space, then the above result holds for any (essentially) bounded measurable map $A:\Omega\rightarrow \sL_{1,\infty}$. Furthermore, for such maps and in the special case of Dixmier traces associated with medial limits, this result is an immediate consequence of the fundamental property of medial limits alluded above (see Section~\ref{sec:Lebesgue}). This confirms Connes' suggestion. 

As it turns out, there are numerous positive traces on $\sL_{1,\infty}$ that are not Dixmier traces. Therefore, it is natural to consider a notion of measurability with respect to all positive traces on $\sL_{1,\infty}$ (see, e.g.,~\cite{KLPS:AIM13, LSZ:Book, SSUZ:AIM15}). We shall call such operators \emph{strongly measurable}. For sake of completeness we overview their main properties. These operators actually form a natural domain for the NC integral, in the sense that its restriction to strongly measurable still satisfies the NC integral's Ansatz (see Proposition~\ref{prop:NCInt.strong-measurable}). We also establish the spectral invariance of the strong measurability property (see Proposition~\ref{prop:NCint.spectral-inv.strong}). 

Strong measurability naturally appear in the context of Connes' trace theorem~\cite{Co:CMP88, KLPS:AIM13}. Suppose that $(M^n,g)$ is a closed Riemannian manifold and $E$ is a Hermitian vector bundle over $M$. Recall that Connes' trace theorem asserts that if $P:C^\infty(M,E)\rightarrow C^\infty(M,E)$ is a pseudodifferential operator (\psido)  of order~$-n$, then $P$ is strongly measurable, and we have
\begin{equation}
 \bint P = \frac1{n} \int_{S^*M} \tr_E\big[\sigma(P)(x,\xi)\big]dxd\xi,
 \label{eq:Intro.Connes-trace-thm}
\end{equation}
where $\sigma(P)$ is the principal symbol of $P$ and $S^*M$ is the cosphere bundle equipped with its Liouville measure $dxd\xi$. 
 The r.h.s.~is the noncommutative residue trace of $P$ in the sense of Guillemin~\cite{Gu:AIM85} and Wodzicki~\cite{Wo:NCRF}. Applying the above result to $P=f\Delta_g^{-n/2}$, where $f\in C^\infty(M)$ and $\Delta_g$ is the Laplace-Beltrami operator on functions, gives Connes' integration formula,
\begin{equation}
 \bint f\Delta_g^{-\frac{n}2} = c_n \int_M f(x) \sqrt{g(x)}dx, \qquad c_n:= \frac{1}{n} (2\pi)^{-n}\big|\bS^{n-1}\big|. 
 \label{eq:Intro.Connes-Integration-Formula} 
\end{equation}
This shows that the NC integral recaptures the Riemannian measure $\sqrt{g(x)}dx$. 

In the special case $f=1$ Connes' integration formula~(\ref{eq:Intro.Connes-Integration-Formula}) is an immediate consequence of the Weyl's law for the Laplacian $\Delta_g$. Strong measurability does not imply Weyl's law. Therefore, we are lead to the following question: 

\begin{questionC}
 What is the precise relationship between Weyl's law and measurability? 
\end{questionC}

This question is closely related to the work of Birman-Solomyak~\cite{BS:JFAA70, BS:VLU77, BS:VLU79, BS:SMJ79} on Weyl's laws for compact operators in the 70s. This work was partly motivated by the semiclassical analysis of Schr\"odinger operators, since by the Birman-Schwinger principle  Weyl's laws for compact operators yield semiclassical Weyl's laws for Schr\"odinger operators. In particular, Birman-Solomyak~\cite{BS:JFAA70} set-up a full perturbation theory. In~\cite{BS:VLU77, BS:VLU79, BS:SMJ79} they further showed that if $P$ is a selfadjoint \psido\ of order $-m<0$, and we set $p=nm^{-1}$, then we have the following Weyl's law,  
\begin{equation}
 \lim_{j\rightarrow \infty} j^{\frac1p} \lambda^\pm_j(P)= \bigg[\frac1{n} (2\pi)^{-n} \int_{S^*M} \tr_E\big[ \sigma(P)(x,\xi)_\pm^{p} \big] dx d\xi\bigg]^{\frac1{p}}.
 \label{eq:Intro.Bir-Sol-selfadjoint}
\end{equation}
where $\pm\lambda_{0}^\pm(P) \geq \pm\lambda_{1}^\pm(P)\geq \ldots $ are the positive and negative eigenvalues of $P$. We also have a similar results for the singular values of $P$ (i.e., the eigenvalues of $|P|=\sqrt{P^*P}$) without any selfadjointness assumption. Namely, 
\begin{equation}
 \lim_{j\rightarrow \infty} j^{\frac1p} \mu_j(P)= \bigg[\frac1{n} (2\pi)^{-n} \int_{S^*M} \tr_E\big[ |\sigma(P)(x,\xi)|^{p} \big] dx d\xi\bigg]^{\frac1{p}}, 
 \label{eq:Intro.Bir-Sol-|P|}
\end{equation}
where $\mu_0(P)\geq \mu_1(P)\geq \cdots$ are the singular values of $P$. 

As it turns out (see Proposition~\ref{prop:Bir-Sol.Weyl-mesurable}), any selfadjoint operator $A\in \sL_{1,\infty}$ satisfying a Weyl's law of the form~(\ref{eq:Intro.Bir-Sol-selfadjoint}) for $p=1$ is \emph{strongly} measurable, and we have
\begin{equation}
 \bint A =  \lim_{j\rightarrow \infty} j \lambda^+_j(A) -  \lim_{j\rightarrow \infty} j \lambda^-_j(A).
 \label{eq:Intro.bint-lambdapm}
\end{equation}
There is a similar result for the absolute value $|A|$ in terms of the singular values of $A$ (\emph{cf.}~Corollary~\ref{cor:Bir-Sol.Weyl-mesurable-|A|}). This answers Question~C. This also provides a further spectral theoretic description of the NC integral for operators satisfying Weyl's laws. Incidentally, this shows that the Weyl's laws~(\ref{eq:Intro.Bir-Sol-selfadjoint})--(\ref{eq:Intro.Bir-Sol-|P|}) of Birman-Solomyak provide us with a stronger form of Connes' trace theorem~(\ref{eq:Intro.Connes-trace-thm}). For instance, if $P$ is any \psido\ of order $-n$, then its absolute value $|P|$ is strongly measurable, even though it need not be a \psido.

The original proof of the Weyl's laws~(\ref{eq:Intro.Bir-Sol-selfadjoint})--(\ref{eq:Intro.Bir-Sol-|P|}) by Birman-Solomyak~\cite{BS:JFAA70, BS:VLU77, BS:VLU79, BS:SMJ79} is arguably a beautiful piece of hard analysis. Unfortunately, the main key technical details are exposed in a somewhat compressed manner 
in the Russian article~\cite{BS:VLU79}, the translation of which remains unavailable. 

\begin{questionD}
 Is there a ``soft proof'' of Birman-Solomyak's Weyl's laws~(\ref{eq:Intro.Bir-Sol-selfadjoint})--(\ref{eq:Intro.Bir-Sol-|P|})? 
\end{questionD}

We provide such a proof in Section~\ref{sec:Weyl-neg-PsiDOs}. The approach uses the relationship between zeta functions and the noncommutative residue to get the Weyl's laws~(\ref{eq:Intro.Bir-Sol-selfadjoint})--(\ref{eq:Intro.Bir-Sol-|P|}) for inverses of elliptic operators. The perturbation theory of Birman-Solomyak~\cite[\S4]{BS:JFAA70} and the BKS inequality~\cite{BKS:IVUZM75} then allow us to get the Weyl's laws in the general case. We refer to Section~\ref{sec:Weyl-neg-PsiDOs} for the full details.

Noncommutative geometry and semiclassical analysis are usually considered to be different sub-fields of quantum theory. As mentioned above the Birman-Schwinger principle provides a bridge between Weyl's laws for compact operators and semiclassical Weyl's laws for Schr\"odinger operators. As we have related the former to Connes' integration, we are lead to the following question:

\begin{questionE}
 What is the precise relationship between Connes' integral and semiclassical Weyl's laws for Schr\"odinger operators?
\end{questionE}

It's clear that an answer is provided by the Birman-Schwinger principle  and the integration formula~(\ref{eq:Intro.bint-lambdapm}). The question is more like to determine the level of generality at which the relationship holds and explain how simple this relation is (compare~\cite{MSZ:arXiv21}). 

To wit let $H$ be a bounded from below operator with non-negative spectrum. We assume that $0$ is an isolated eigenvalue with finite multiplicity. Let $V$ be a selfadjoint $H$-form compact perturbation. The operator $H+V$ then makes sense as a form sum and the negative part of its spectrum is discrete. Let $N^{-}(H+V)$ be the number of its negative eigenvalues counted with multiplicity (i.e., the number of bound states in physics' jargon). 

The abstract Birman-Schwinger principle~\cite{BS:AMST89} (see also~\cite{MP:Part1}) relates $N^{-}(H+V)$ to the counting function of the Birman-Schwinger operator $H^{-1/2}VH^{-1/2}$. Combining this with the integration formula~(\ref{eq:Intro.bint-lambdapm}) shows that if $V\geq 0$ and $H^{-1/2}VH^{-1/2}$ satisfies a Weyl's law of the form~(\ref{eq:Intro.Bir-Sol-selfadjoint}) for $p=1$, then (see Proposition~\ref{prop:SC.NHV-bint}) we have 
\begin{equation*}
   \lim_{h \rightarrow 0^+} h^2 N^{-}\big( h^2H-V\big)=  \bint H^{-\frac12}VH^{-\frac12}.
\end{equation*}
This answers Question~E and gives a neat link between Connes' noncommutative geometry and the semiclassical analysis of Schr\"odinger operators. We also illustrate this result in terms of the recent Weyl's laws for $\LlogL$-Orlicz potentials on closed manifolds by Rozenblum~\cite{Ro:arXiv21} and Sukochev-Zanin~\cite{SZ:arXiv21} (see also~\cite{Po:Weyl-Orlicz, RS:EMS21}). This leads us to a semiclassical interpretation of Connes' integration formula~(\ref{eq:Intro.Connes-Integration-Formula}) (see Section~\ref{sec:SC}). 

The remainder of this paper is organized as follows. In Section~\ref{sec:NCInt}, we deal with Question~A and give a purely spectral theoretic construction of Connes' integral. In Section~\ref{sec:Lebesgue}, we deal with Question~B. In Section~\ref{sec:strongly-measurable}, we describe the main properties of strongly measurable operators. 
In Section~\ref{sec:Weyl}, we deal with Question~C by relating Connes' integration to the Weyl's laws for compact operators.  In Section~\ref{sec:Bir-Sol-PDOs}, we give a ``soft proof'' of Birman-Solomyak's Weyl's laws~(\ref{eq:Intro.Bir-Sol-selfadjoint})--(\ref{eq:Intro.Bir-Sol-|P|}); this deals with Question~D. In Section~\ref{sec:SC}, we deal with Question~E by explaining the link between Connes' integral and semiclassical Weyl's laws. 
Finally, in Appendix~\ref{sec:Appendix}, we gather a few results on Hilbert spaces embeddings that are needed in Section~\ref{sec:NCInt}.


\section{Quantized Calculus and NC Integral}\label{sec:NCInt} 
In this section, we present a purely spectral theoretic construction of Connes' integral. 
After a brief review of weak Schatten classes and Connes' quantized calculus, we give two constructions of the NC integral. The first one is given in terms of Dixmier traces and uses extended limits. The other construction is in terms of Tauberian operators. These two constructions are shown to give exactly the same notion of NC integral. This will answer Question~A.

\subsection{Weak Schatten Classes} \label{subsec:schatten}
We briefly review the main definitions and properties regarding Schatten and weak Schatten classes (see, e.g.,~\cite{Si:AMS05, GK:AMS69} for further details). 

Throughout this paper we let $\sH$ be a (separable) Hilbert space with inner product $\scal{\cdot}{\cdot}$. The algebra of  bounded linear operators on $\sH$ is denoted $\sL(\sH)$. The operator norm is denoted $\|\cdot\|$. We also denote by $\sK$ the (closed) ideal of compact operators on $\sH$. Given any operator $T\in \sK$ we let $(\mu_j(T))_{j\geq 0}$ be its sequence of \emph{singular values}, i.e., $\mu_j(T)$ is the $(j+1)$-th eigenvalue counted with multiplicity of the absolute value $|T|=\sqrt{T^*T}$. The \emph{min-max principle} states that
\begin{align}
 \mu_j(T)&=\min \left\{\|T_{|E^\perp}\|;\ \dim E=j\right\}.
 \label{eq:min-max} 
\end{align}
We record the following properties of singular values (see, e.g., \cite{GK:AMS69, Si:AMS05}), 
\begin{gather}
 \mu_j(T)=\mu_j(T^*)=\mu_j(|T|),
 \label{eq:Quantized.properties-mun1}\\
 \mu_{j+k}(S+T)\leq \mu_j(S) + \mu_k(T),
 \label{eq:Quantized.properties-mun2}\\
 \mu_j(ATB)\leq \|A\| \mu_j(T) \|B\|, \qquad A, B\in \sL(\sH), 
 \label{eq:Quantized.properties-mun3}.
\end{gather}
The inequality~(\ref{eq:Quantized.properties-mun2}) is known as Ky Fan's inequality.  

For $p\in (0,\infty)$ the Schatten class $\sL_p$ consist of operators $T\in \sK$ such that $|T|^p$ is trace-class. It is equipped with the quasi-norm, 
\begin{equation*}
 \|T\|_{p}:=\Tr\big(|T|^p\big)^{\frac1p}= \bigg( \sum_{j\geq 0}\mu_j(T)^p\bigg)^{\frac1p}, \qquad T\in \sL_p. 
\end{equation*}
We obtain a quasi-Banach ideal. For $p\geq 1$ the $\sL_p$-quasi-norm is actually a norm, and so in this case $\sL_p$ is a Banach ideal. In any case, the finite-rank operators on $\sH$ form a dense subspace of $\sL_p$. 

For $p\in (0,\infty)$, the weak Schatten class  $\sL_{p,\infty}$ is defined by
\begin{equation*}
 \sL_{p,\infty}:=\left\{T\in \sK; \ \mu_j(T)=\op{O}\big(j^{-\frac1p}\big)\right\}. 
\end{equation*}
 This is a two-sided ideal. We equip it with the quasi-norm,
\begin{equation}\label{def:lp_infty_quasinorm}
 \|T\|_{p,\infty}:=\sup_{j\geq 0}\;(j+1)^{\frac{1}{p}}\mu_j(T), \qquad T\in \sL_{p,\infty}. 
\end{equation}
For $p> 1$, the quasi-norm $\|\cdot \|_{p,\infty}$ is equivalent to the norm, 
\begin{equation*}
 \|T\|_{p,\infty}':=  \sup_{N\geq 1} N^{-1+\frac1{p}}\sum_{j<N} \mu_j(T), \qquad T\in \sL_{p,\infty}.
\end{equation*}
Thus, in this case $\sL_{p,\infty}$ is a Banach ideal with respect to that equivalent norm. In general (see, e.g., \cite{Si:TAMS76}), we have 
\begin{equation}
 \|S+T\|_{p,\infty}\leq 2^{\frac1p}\left(\|S\|_{p,\infty} + \|T\|_{p,\infty}\right), \qquad S,T\in \sL^{p,\infty}.
 \label{eq:Quantized.quasi-norm} 
\end{equation}
In addition, we denote by $(\sL_{p,\infty})_{0}$ the closure in $\sL_{p,\infty}$ of the finite-rank operators. We have
\begin{equation*}
 \big(\sL_{p,\infty}\big)_{0}=\left\{T\in \sK; \ \mu_j(T)=\op{o}\big(j^{-\frac1p}\big)\right\}.
\end{equation*}
We note the continuous inclusions, 
\begin{equation*}
 \sL_p \subsetneq  \big(\sL_{p,\infty}\big)_{0}\subsetneq \sL_{p,\infty}  \subsetneq \sL_{q}, \qquad 0<p<q.
\end{equation*}

In the following, we will also denote the Schatten and weak Schatten classes by $\sL_p(\sH)$ and $\sL_{p,\infty}(\sH)$ whenever there is a need to specify the Hilbert space. 

\subsection{Quantized calculus} \label{subsec:quantised_calculus}
 The main goal of the quantized calculus of Connes~\cite{Co:NCG} is to translate into the Hilbert space formalism of quantum mechanics the main tools of the classical infinitesimal calculus. 

\renewcommand{\arraystretch}{1.2}

\begin{center}
    \begin{tabular}{c|c}  
        Classical & Quantum \\ \hline       
       Complex variable & Operator on $\sH $  \\
      Real variable &  Selfadjoint operator on $\sH $  \\  
 Infinitesimal variable & Compact operator on $\sH $ \\
       Infinitesimal of order $\alpha>0$  & Compact operator $T$ such that\\ 
                      &  $\mu_{j}(T)=\op{O}(j^{-\alpha})$\\
    \end{tabular}
\end{center}

The first two lines arise from quantum mechanics. Intuitively speaking, an infinitesimal is meant to be smaller than any real number. For a bounded operator the condition $\|T\|<\epsilon$ for all $\epsilon>0$ gives $T=0$. This condition can be relaxed into the following: For every $\epsilon>0$ there is a finite-dimensional subspace $E$ of 
$\sH$ such that $\|T_{|E^\perp}\|<\epsilon$. This is equivalent to $T$ being a compact operator.

The order of compactness of a compact operator is given by the order of decay of its singular values. Namely, an \emph{infinitesimal operator} of order $\alpha>0$ is any compact operator such that $\mu_j(T)=\op{O}(j^{-\alpha})$. Thus, if we set $p=\alpha^{-1}$, then $T$ is {infinitesimal operator} of order $\alpha>0$ iff $T\in \sL_{p,\infty}$. 

The next line of the dictionary is the NC analogue of the integral. As an Ansatz the NC integral should be a linear functional satisfying at least the following conditions: 
\begin{enumerate}
 \item[(1)] It is defined on a suitable class of infinitesimal operators of order~1. 
 
 \item[(2)] It vanishes on infinitesimal operators of order~$> 1$. 
 
 \item[(3)] It takes non-negative values on positive operators. 
 
 \item[(4)] It is invariant under Hilbert space isomorphisms.  
 \end{enumerate}
As mentioned above, the infinitesimal operators of order~1 are the operators in the weak trace class $\sL_{1,\infty}$. The condition (3) means that the functional should be positive. The condition (4) forces the functional to be a trace, in the sense it is annihilated by the commutator subspace, 
\begin{equation*}
 \op{Com}\big(\sL_{1,\infty}\big):=\op{Span}\left\{[A,T];\ A\in \sL(\sH), \ T\in\sL_{1,\infty}\right\}. 
\end{equation*}
More precisely, it would be convenient to adopt the following definition of a trace. 

\begin{definition}
 If $\sE$ is a subspace of $\sL_{1,\infty}$ containing $\Com(\sL_{1,\infty})$ and $\sF$ is another vector space, then we say that a linear map $\varphi:\sE\rightarrow \sF$ is a \emph{trace} if it is annihilated by $\Com(\sL_{1,\infty})$. 
\end{definition}

To sum up the NC integral should be a positive trace $\bint: \sM\rightarrow \C$, where $\sM$ is a suitable subspace of $\sL_{1,\infty}$ containing the commutator subspace $\op{Com}(\sL_{1,\infty})$ and infinitesimal operators of order~$>1$.

\subsection{Eigenvalue sequences and commutators in $\sL_{1,\infty}$} 
If $A$ is a compact operator on $\sH$, then its spectrum can be arranged as a sequence $(\lambda_j(A))_{j\geq 0}$ converging to $0$ such that
\begin{equation*}
|\lambda_0(A)|\geq |\lambda_1(A)|\geq \cdots \geq  |\lambda_j(A)|\geq \cdots \geq 0,
\end{equation*}
where each eigenvalue is repeating according to its algebraic multiplicity, i.e., the dimension of the root space
$ E_\lambda(A) :=\bigcup_{\ell\geq 0} \ker(A-\lambda)^\ell$.  If $\lambda\neq 0$, the algebraic multiplicity is always finite (see, e.g., \cite{GK:AMS69}). It agrees with the geometric multiplicity whenever $A$ is normal.  

A sequence as above is called an \emph{eigenvalue sequence} for $A$. Such a sequence is not unique. If $A\geq 0$, then the eigenvalue sequence is unique and agrees with its singular value sequence $(\mu_j(T))_{j\geq 0}$. In general, an eigenvalue sequence need not be unique. 

In what follows by $(\lambda_j(A))_{j\geq 0} $ we shall always denote an eigenvalue sequence in the sense above.  

We record the Ky Fan's inequalities (see, e.g., \cite{GK:AMS69, Si:AMS05}). 
\begin{equation}
\big| \sum_{j<N} \lambda_j(A) \big| \leq  \sum_{j<N} \left|\lambda_j(A) \right|  \leq  \sum_{j<N} \mu_j(A) \qquad \forall N\geq 1. 
\label{eq:NC-Integral.Weyl-Ineq} 
\end{equation}

The approach of this section is based on the following asymptotic additivity result. 

\begin{lemma}[{\cite[Lemma~5.7.5]{LSZ:Book}}] \label{lem:NCInt.additivity}
 If $A$ and $B$ are operators in $\sL_{1,\infty}$, then
\begin{equation}
 \sum_{j<N} \lambda_j(A+B) = \sum_{j<N} \lambda_j(A)  + \sum_{j<N} \lambda_j(B) +\op{O}(1). 
 \label{eq:NCInt.additivity} 
\end{equation}
\end{lemma}

\begin{remark}\label{rmk:NCint.two-eig-seq}
 For $B=0$ the above result shows that if $(\lambda_j(A))_{j\geq 0}$ and $(\lambda_j'(A))_{j\geq 0}$ are two eigenvalue sequences of $A$, then 
\begin{equation}
 \sum_{j<N} \lambda_j'(A) = \sum_{j<N} \lambda_j(A) + \op{O}(1). 
 \label{eq:NCInt.additivity-A}
\end{equation} 
\end{remark}

\begin{remark}[see~{\cite[Lemma~5.7.1]{LSZ:Book}}] \label{rmk:NCInt.self-eig}
Suppose that $A=A^*\in \sL_{1,\infty}$. Let $(\pm\lambda_j^{\pm}(A))_{j\geq 0}$ be the sequence of positive/negative eigenvalues of $A$. That is, $\lambda_j^\pm(A)=\lambda_j(A^\pm)=\mu_j(A^\pm)$, where $A^{\pm}=\frac12(|A|\pm A)$ are the positive and negative parts of $A$. As $A=A^+-A^{-}$ we get
\begin{equation*}
 \sum_{j<N} \lambda_j(A) = \sum_{j<N} \big(\lambda_j^+(A)-\lambda^{-}_j(A)\big) + \op{O}(1). 
\end{equation*}
\end{remark}

\begin{remark}\label{rmk:NCInt.real-imaginary-parts} 
 Given $A\in \sL_{1,\infty}$, let $\Re A=\frac12(A+A^*)$ and $\Im A= \frac1{2i}(A-A^*)$ be its real and imaginary parts. Then we have 
 \begin{equation*}
  \sum_{j<N} \lambda_j(A) = \sum_{j<N} \big(\lambda_j(\Re A)+i  \lambda_j(\Im A)\big)  +\op{O}(1). 
\end{equation*}
As $\lambda_j(\Re A)$ and  $\lambda_j(\Im A)$ are real numbers, we get 
\begin{gather*}
 \Re\big(\sum_{j<N} \lambda_j(A) \big)=  \sum_{j<N} \Re\left(\lambda_j(A) \right)= \sum_{j<N} \lambda_j(\Re A) +\op{O}(1),\\
 \Im\big(\sum_{j<N} \lambda_j(A) \big)=  \sum_{j<N} \Im\left(\lambda_j(A) \right)= \sum_{j<N} \lambda_j(\Im A) +\op{O}(1). 
\end{gather*}
\end{remark}

We have the following consequence of Lemma~\ref{lem:NCInt.additivity}. 

\begin{corollary}\label{cor:NCint.Com}
 If $A\in \Com(\sL_{1,\infty})$, then
 \begin{equation*}
 \sum_{j<N} \lambda_j(A)=\op{O}(1). 
\end{equation*}
 \end{corollary}
\begin{proof}
 By definition $\Com(\sL_{1,\infty})$ is spanned by commutators  of the form $[T,A]$ with $T\in \sL(\sH)$ and $A\in \sL_{1,\infty}$. Any $T\in \sL(\sH)$ is a linear combination of unitary operators (see, e.g., \cite[Section~VI.6]{RS1:1980}). Thus, $\Com(\sL_{1,\infty})$ is spanned by operators of the form $[U,A]=UA-U^*(UA)U$ with $A\in \sL_{1,\infty}$ and $U\in \sL(\sH)$ unitary. Combining with the asymptotic additivity~(\ref{eq:NCInt.additivity}) of sums of eigenvalues, we see that it is enough to prove the result for operators of the form $A=U^*BU-U$ with $B\in \sL_{1,\infty}$ and $U\in \sL_{1,\infty}$ unitary. However, in this case any eigenvalue sequence for $B$ is an eigenvalue sequence for 
 $U^*BU$. Therefore, by using~(\ref{eq:NCInt.additivity}) we get 
 \begin{equation*}
 \sum_{j<N} \lambda_j(A)= \sum_{j<N} \lambda_j\left(U^*BU\right) - \sum_{j<N} \lambda_j(B)+\op{O}(1)= \op{O}(1). 
\end{equation*}
The proof is complete. 
\end{proof}

We have a converse to Corollary~\ref{cor:NCint.Com}. More precisely, we have the following result. 

\begin{proposition}[\cite{DFWW:AIM04, LSZ:Book}] \label{prop:NCint.Com-DFWWW}
If $A$ and $B$ are operators in $\sL_{1,\infty}$, then 
 \begin{equation}
 A-B\in  \Com\big(\sL_{1,\infty}\big) \Longleftrightarrow  \sum_{j<N} \lambda_j(A)=\sum_{j<N} \lambda_j(B)+\op{O}(1).
 \label{eq:NCInt.A-B-Com}  
\end{equation}
\end{proposition}

\begin{remark}
 Proposition~\ref{prop:NCint.Com-DFWWW} is a special case of a deep characterization of the commutator spaces of compact operator ideals due to Dykema-Figiel-Weiss-Wodzicki~\cite{DFWW:AIM04}. However, in the special case of $\sL_{1,\infty}$ the proof is much simpler (see~\cite[\S5.7]{LSZ:Book}). 
\end{remark}

\subsection{The noncommutative integral in terms of Dixmier traces}\label{subsec:NCInt.Dixmier}
We shall now recall the construction of the NC integral in terms of Dixmier traces. Our construction deviates a bit from the standard constructions of Dixmier~\cite{Di:CRAS66} and Connes~\cite{Co:NCG}, since we work on the weak trace class, rather than the Dixmier-Macaev ideal. The approach is solely based on using the asymptotic additivity property provided by Lemma~\ref{lem:NCInt.additivity}. 
The exposition is also partly inspired by the construction of Dixmier traces by Connes-Moscovici~\cite[Appendix A]{CM:GAFA95}. 

In what follows we denote by $\ell_\infty$ the $C^*$-algebra of bounded sequences $(a_N)_{N\geq 1}\subset \C$. We also let $\ell_0$ be the closed ideal of sequences converging to $0$. We then endow the quotient $\ell_\infty/\ell_0$ with its quotient $C^*$-algebra structure. 

If $A\in \sL_{1,\infty}$, the Ky Fan's inequalities~(\ref{eq:NC-Integral.Weyl-Ineq}) imply that
\begin{equation}
\bigg| \frac1{\log N} \sum_{j<N} \lambda_j(A)\bigg| \leq  \frac1{\log N} \sum_{j<N} \mu_j(A)\leq C\|A\|_{1,\infty}, 
\label{eq:NCInt.Weyl-norm-ineq}  
\end{equation}
where the constant $C$ does not depend on $A$. (We make the convention that $(\log N)^{-1}=0$ for $N=1$). Thus, the sequence $\{(\log N)^{-1} \sum_{j<N} \lambda_j(A)\}_{N\geq 1}$ is bounded. Moreover, it follows from Remark~\ref{rmk:NCint.two-eig-seq} that, if $(\lambda_j'(A))_{j\geq 0}$ is another eigenvalue sequence for $A$, then 
\begin{equation*}
 \frac1{\log N} \sum_{j<N} \lambda_j'(A)=  \frac1{\log N} \sum_{j<N} \lambda_j(A) +\op{o}(1). 
\end{equation*}
Thus, the class of  $\{\frac1{\log N} \sum_{j<N} \lambda_j(A)\}_{N\geq 1}$ in  $\ell_\infty/\ell_0$ does not depend on the choice of the eigenvalue sequence 
$(\lambda_j(A))_{j\geq 0}$. Therefore, we have a well defined map $\tau:\sL_{1,\infty} \rightarrow \ell_\infty/\ell_0$ given by 
\begin{equation*}
 \tau(A)= \textup{class of}\ \bigg\{\frac1{\log N} \sum_{j<N} \lambda_j(A)\bigg\}_{N\geq 1}\ \textup{in}\  \ell_\infty/\ell_0. 
\end{equation*}

\begin{lemma}\label{lem:NCInt.tau}
 The map $\tau:\sL_{1,\infty} \rightarrow \ell_\infty/\ell_0$  is a positive continuous linear trace.  It is  annihilated by operators in $(\sL_{1,\infty})_0$, including infinitesimal operators of order~$>1$. 
\end{lemma}
\begin{proof}
 Let $A\in \sL_{1,\infty}$, and let $(\lambda_j(A))_{j\geq 0}$ be an eigenvalue sequence. If $c\in \C$, then $(c\lambda_j(A))_{j\geq 0}$ is an eigenvalue sequence for $cA$, and hence 
\begin{equation*}
 \tau(cA)=  \textup{class of}\ \bigg\{\frac1{\log N} \sum_{j<N} c\lambda_j(A)\bigg\}_{N\geq 1} =  c\tau(A). 
\end{equation*}
 If $B\in \sL_{1,\infty}$, then it follows from Lemma~\ref{lem:NCInt.additivity} that
\begin{align*}
\frac1{\log N} \sum_{j<N} \lambda_j(A+B)&= \frac1{\log N} \sum_{j<N} \lambda_j(A) + \frac1{\log N} \sum_{j<N} \lambda_j(B)+\op{o}(1). 
\end{align*}
Thus, $\tau(A+B)=\tau(A)+\tau(B)$.  In addition, it follows from~(\ref{eq:NCInt.Weyl-norm-ineq}) that
\begin{equation}
 \|\tau(A)\| \leq \sup_{N\geq 1} \frac1{\log N} \bigg|\sum_{j<N} \lambda_j(A)\bigg| \leq C\|A\|_{1,\infty}.
 \label{eq:NCInt.continuity-tau} 
\end{equation}
Therefore, we see that $\tau$ is a continuous linear map. 

It is immediate that if $A$ has non-negative eigenvalues, then $\tau(A)$ is a positive element of $\ell_\infty/\ell_0$. Thus, $\tau$ is a positive linear map. Furthermore, it follows from Corollary~\ref{cor:NCint.Com} that if $A\in \Com(\sL_{1,\infty})$, then $(\log N)^{-1}\sum_{j<N} \lambda_j(A)$ is $\op{o}(1)$, and hence $\tau(A)=0$. Thus, $\tau$ is a trace. Likewise, if $A\in (\sL_{1,\infty})_0$, then $(\log N)^{-1}\sum_{j<N} \lambda_j(A)=\op{o}(1)$, and hence $\tau(A)=0$. Thus, $\tau$ is annihilated by $(\sL_{1,\infty})_0$. The proof is complete. 
\end{proof}

Recall that a state on unital $C^*$-algebra $\sA$ is a positive linear  functional $\omega:\sA\rightarrow \C$ such that $\omega(1)=1$. Every state on $\sA$ is continuous. Moreover, it follows from the Hahn-Banach theorem that the states separate the points of $\sA$. If $\omega$ is a state on the quotient $C^*$-algebra $\ell_\infty/\ell_0$, then it lifts to a state $\lim_\omega:\ell_\infty \rightarrow \C$ which annihilates $\ell_0$. Namely, $\lim_\omega=\omega\circ \pi$, where $\pi:\ell_\infty \rightarrow \ell_\infty/\ell_0$ is the canonical projection. Such a state is called an \emph{extended limit}. Conversely, any extended limit uniquely descends to a state on $\ell_\infty/\ell_0$. Therefore, we have a one-to-one correspondance between extended limits and states on $\ell_\infty/\ell_0$.

If $\lim_\omega$ is an extended limit, then its positivity implies that, for every sequence $a=(a_N)_{N\geq 1}\in \ell_\infty$, we have
\begin{equation*}
 \liminf a_N \leq \limw a \leq \limsup a_N. 
\end{equation*}
Furthermore, as the states on $\ell_\infty/\ell_0$ form a separating family of linear functionals, we have 
\begin{equation}
\lim_{N\rightarrow \infty} a_N=L \Longleftrightarrow \big( a- L\in \ell_0\big) \Longleftrightarrow  \big( \limw a = L\quad \forall \omega \big).
\label{eq:NCInt.extended-limit} 
\end{equation}

Given any extended limit $\lim_\omega$ we define $\Tr_\omega:\sL_{1,\infty} \rightarrow \C$ by
\begin{equation*}
  \Trw(A) = \limw \frac{1}{\log N} \sum_{j<N} \lambda_j(A), \qquad A \in \sL_{1,\infty}. 
\end{equation*}
Thus, if $A\in \sL_{1,\infty}$, then we have 
\begin{equation*}
 \Trw(A)= \omega \circ \pi\bigg[ \bigg\{ \frac{1}{\log N} \sum_{j<N} \lambda_j(A)\bigg\}_{N\geq 1}\bigg] =\omega\big[\tau(A)\big].
\end{equation*}
Therefore, in view Lemma~\ref{lem:NCInt.tau} we immediately obtain the following result. 

\begin{proposition}\label{prop:NCInt.Dixmier-trace} 
$\Tr_\omega:\sL_{1,\infty} \rightarrow \C$ is a positive continuous linear trace. It is annihilated by operators in $(\sL_{1,\infty})_0$, including infinitesimal operators of order~\mbox{$>1$}.
\end{proposition}

\begin{definition}
 The trace $\Tr_\omega:\sL_{1,\infty} \rightarrow \C$ is called the \emph{Dixmier trace} associated with the extended limit $\lim_\omega$. 
\end{definition}

Every Dixmier trace satisfies the Ansatz for the NC integral. However, if $A\in \sL_{1,\infty}$, the value of $\Tr_\omega(A)$ may depend on the choice the extended limit. To remedy this we proceed as follows. 

\begin{definition}[Connes~\cite{Co:NCG}] 
An operator $A\in \sL_{1,\infty}$ is called \emph{measurable} if the value of  $\Tr_\omega(A)$ is independent of the choice of the extended limit. For such an operator,  its \emph{NC integral} is defined by
\begin{equation*}
 \bint A = \limw \frac{1}{\log N} \sum_{j<N} \lambda_j(A), 
\end{equation*}where $\lim_\omega$ is any extended limit.   
\end{definition}

In what follows we denote by $\sM$ the set of measurable operators. The NC integral then is a map $\bint: \sM\rightarrow \C$. 

\begin{proposition}
 The following holds.
\begin{enumerate}
\item $\sM$ is a closed subspace of $\sL_{1,\infty}$ containing $\Com(\sL_{1,\infty})$ and $(\sL_{1,\infty})_0$.  In particular, all infinitesimal operators of order~$>1$ are measurable. 

\item The NC integral $\bint: \sM\rightarrow \C$ is a positive continuous trace on $\sM$. It is annihilated by operators in $(\sL_{1,\infty})_0$, including infinitesimal operators of order~$>1$.
\end{enumerate}
\end{proposition}
\begin{proof}
 By definition, 
 \begin{equation*}
 \sM=\bigcap_{\omega,\omega'} \big\{A\in \sL_{1,\infty};\ \Tr_\omega(A)=\Tr_{\omega'}(A)\big\},
\end{equation*}
where $\omega$ and $\omega'$ range over all states on $\ell_\infty/\ell_0$. As the Dixmier traces $\Tr_\omega$ are continuous linear maps, it follows that $\sM$ is a closed subspace of $\sL_{1,\infty}$. 

By definition the NC integral $\bint$ agrees with any Dixmier trace $\Tr_\omega$, and so this is a continuous positive linear functional by Proposition~\ref{prop:NCInt.Dixmier-trace}. Moreover, as the union $\Com(\sL_{1,\infty}) \cup (\sL_{1,\infty})_0$ is annihilated by every Dixmier trace,  it is contained in $\sM$ and is annihilated by $\bint$. In particular, the NC integral $\bint$ is a trace on $\sM$. The proof is complete. 
 \end{proof}

\subsection{A noncommutative integral in terms of Tauberian operators} We shall now present an alternative approach to the NC integral. The approach is to work with Tauberian operators (see definition below). This approach is inspired by the characterization of measurable operators in terms of Tauberian operators in~\cite{LSZ:Book}. 

We will show in the next subsection that this approach is equivalent to the previous approach in terms of Dixmier approach.  As the 2nd approach involves spectral data only, the equivalence between the two approaches will answer Question~A. 

If $A\in \sL_{1,\infty}$, then it follows from~(\ref{eq:NCInt.additivity-A}) that if $(\lambda_j(A))_{j\geq 0}$ and $(\lambda_j'(A))_{j\geq 0}$ are two eigenvalue sequences for $A$, then 
\begin{equation*}
 (\log N)^{-1}\sum_{j<N}\lambda_j'(A)=(\log N)^{-1}\sum_{j<N}\lambda_j(A)+\op{o}(1). 
\end{equation*}
This immediately implies the following statement. 

\begin{lemma}\label{lem:NCInt.Tauberian-any}
 Given $A\in \sL_{1,\infty}$ and $L\in \C$, the following are equivalent:
 \begin{enumerate}
 \item[(i)]  $(\log N)^{-1}\sum_{j<N}\lambda_j(A)\rightarrow L$ for some eigenvalue sequence $(\lambda_j(A))_{j\geq 0}$ of $A$. 
 
\item[(ii)] $(\log N)^{-1}\sum_{j<N}\lambda_j(A)\rightarrow L$ for every eigenvalue sequence $(\lambda_j(A))_{j\geq 0}$ of $A$. 
\end{enumerate}
\end{lemma}

\begin{definition}[see, e.g., \cite{LSZ:Book}] 
Any operator $A\in \sL_{1,\infty}$ that satisfies the conditions of Lemma~\ref{lem:NCInt.Tauberian-any} is called a \emph{Tauberian operator}. 
\end{definition}

In what follows, we  denote by $\sT$ the class of Tauberian operators in $\sL_{1,\infty}$. If $A \in \sT$, we set
\begin{equation*}
 \bint'A:= \lim_{N\rightarrow \infty}  \frac{1}{ \log N}\sum_{j<N}\lambda_j(A), 
\end{equation*}
where $(\lambda_j(A))_{j\geq 0}$ is any eigenvalue sequence for $A$. Thanks to Lemma~\ref{lem:NCInt.Tauberian-any} the above limit exists for any eigenvalue sequence and its value is independent of the choice of that sequence. 

The following result shows that the map $\bint':\sT\rightarrow \C$ has all the properties we are seeking for the NC integral. 

\begin{proposition}\label{prop:NCInt.properties-Lambda} 
 The following holds. 
\begin{enumerate}
 \item $\sT$ is a subspace of $\sL_{1,\infty}$ containing $\Com(\sL_{1,\infty})$ and $(\sL_{1,\infty})_0$. 
 
 \item The map $\bint': \sT\rightarrow \C$ is a continuous positive linear trace on $\sT$. It is  annihilated by  operators in $(\sL_{1,\infty})_0$, including infinitesimals of order~$>1$. 
\end{enumerate}
\end{proposition}
\begin{proof}
It is immediate that if an operator $A\in \sT$ is positive, then $\bint'A\geq 0$.  Moreover, if $A\in \sT$ and $c\in \C$, then $(c\lambda_j(A))_{j\geq 0}$ is an eigenvalue sequence of $cA$, and hence $cA\in \sT$ with $\bint'(cA)=c\bint'(A)$. If $A,B\in \sT$, then~(\ref{eq:NCInt.additivity}) implies  that
\begin{equation*}
 \frac{1}{ \log N}\sum_{j<N}\lambda_j(A+B) = \frac{1}{ \log N}\sum_{j<N}\lambda_j(A) + \frac{1}{ \log N}\sum_{j<N}\lambda_j(B)+\op{o}(1). 
\end{equation*}
Thus, 
\begin{align*}
 \lim_{N\rightarrow \infty}  \frac{1}{ \log N}\sum_{j<N}\lambda_j(A+B) &=  \lim_{N\rightarrow \infty} \frac{1}{ \log N}\sum_{j<N}\lambda_j(A) +  \lim_{N\rightarrow \infty} \frac{1}{ \log N}\sum_{j<N}\lambda_j(B)\\
 & =\bint'A+\bint' B.  
\end{align*}
That is, $A+B\in \sT$ and $\bint'(A+B)=\bint'A+\bint'B$. All this shows that $\sT$ is a subspace of $\sL_{1,\infty}$ and $\bint': \sT\rightarrow \C$ is a positive linear map. 

If $A\in (\sL_{1,\infty})_0$, then $\mu_j(A)=\op{o}(j^{-1})$, and so $\sum_{j<N} \mu_j(A)=\op{o}(\log N)$. Combining this with the Ky Fan's inequality~(\ref{eq:NC-Integral.Weyl-Ineq}) shows that 
$\sum_{j<N} \lambda_j(A)=\op{o}(\log N)$, i.e., $A\in \sT$ and $\bint'A=0$. Likewise, if  $A \in \Com(\sL_{1,\infty})$, then Corollary~\ref{cor:NCint.Com} implies that $\sum_{j<N} \lambda_j(A)$ is $\op{0}(1)$, and hence is $\op{o}(\log N)$. Thus, in this case too, $A\in \sT$ and $\bint'A=0$. In particular, this shows that $\bint'$ is a trace on $\sT$. The proof is complete. 
\end{proof}

\subsection{Equivalence between the two approaches. Spectral invariance} 
We shall now explain that the two approaches to the NC integral coincide. Namely, we have the following result. 

\begin{theorem}\label{thm:NCTint.Tauberian-measurable}
 An operator $A\in \sL_{1,\infty}$ is measurable if and only if it is Tauberian. Moreover, in this case we have
 \begin{equation}
 \bint A =  \lim_{N\rightarrow \infty} \frac{1}{ \log N}\sum_{j<N}\lambda_j(A).
 \label{eq:NCInt.trace-formula}
\end{equation}
\end{theorem}
\begin{proof}
 This is a direct consequence of~(\ref{eq:NCInt.extended-limit}), since it gives
\begin{equation*}
 \bigg( \limw  \frac{1}{ \log N}\sum_{j<N}\lambda_j(A) =L \quad \forall \omega \bigg)
  \Longleftrightarrow   \lim_{N\rightarrow \infty} \frac{1}{ \log N}\sum_{j<N}\lambda_j(A) =L.  
\end{equation*}
 The l.h.s.\ exactly means that $A\in \sM$ and $\bint A=L$. The r.h.s.\ exactly means that $A\in \sT$ and $\bint' A=L$. Hence the result. 
\end{proof}

\begin{remark}
 The characterization of measurable operators in terms of the Tauberian property is the contents of~\cite[Theorem~9.7.5]{LSZ:Book}. The proof given above is somewhat simpler. The trace formula~(\ref{eq:NCInt.trace-formula}) is not established in~\cite{LSZ:Book}.
\end{remark}

Theorem~\ref{thm:NCTint.Tauberian-measurable} characterizes measurable operators and shows how to compute NC integrals purely in terms of spectral data. In particular, the computation of the NC integral of some concrete operator only requires the knowledge of its spectrum. This answers Question~A in the introduction. Incidentally, this shows that $(\sM,\bint)$ depends on the locally convex topology of $\sH$ in a somewhat loose sense. In particular, it does not depend on the choice of the inner product. 

We mention a few consequences of Theorem~\ref{thm:NCTint.Tauberian-measurable}.

\begin{proposition}\label{prop:NCInt.Measurable-Re-Im} 
 Let $A\in \sL_{1,\infty}$. Then $A$ is measurable if and only if its real part $\Re A$ and its imaginary part $\Im A$ are both measurable. Moreover, in this case we have
\begin{equation*}
 \Re\bigg( \bint A \bigg) =\bint \Re A, \qquad  \Im \bigg( \bint A \bigg)  =\bint \Im A. 
\end{equation*}
\end{proposition}
\begin{proof}
 It follows from Remark~\ref{rmk:NCInt.real-imaginary-parts} that 
\begin{gather*}
  \frac{1}{\log N}\Re\big(\sum_{j<N} \lambda_j(A) \big)=  \frac{1}{\log N}\sum_{j<N} \lambda_j(\Re A) +\op{o}(1),\\
  \frac{1}{\log N}\Im\big(\sum_{j<N} \lambda_j(A) \big)=  \frac{1}{\log N}\sum_{j<N} \lambda_j(\Im A) +\op{o}(1). 
\end{gather*}
Thus $\log N^{-1}\sum_{j<N} \lambda_j(A)\rightarrow L$ as $N\rightarrow \infty$ if and only if 
\begin{equation*}
  \lim_{N\rightarrow \infty}\frac{1}{\log N}\sum_{j<N} \lambda_j(\Re A)=\Re L \quad \text{and} \quad    \lim_{N\rightarrow \infty}\frac{1}{\log N}\sum_{j<N} \lambda_j(\Re A)=\Im L.
\end{equation*}
Combining this with Theorem~\ref{thm:NCTint.Tauberian-measurable} gives the result. 
\end{proof}

\begin{proposition}\label{prop:NCInt.Measurable-selfadj}
 Let $A=A^*\in \sL_{1,\infty}$. Then $A$ is measurable if and only if 
 \begin{equation*}
  \lim_{N\rightarrow \infty} \frac1{\log N}\sum_{j<N} (\lambda_j^+(A)-\lambda_j^-(A)) \quad \text{exists}. 
\end{equation*}
 Moreover, in this case we have
\begin{equation*}
 \bint A =  \lim_{N\rightarrow \infty} \frac1{\log N}\sum_{j<N} (\lambda_j^+(A)-\lambda_j^-(A)). 
\end{equation*}
\end{proposition}
\begin{proof}
 It follows from Remark~\ref{rmk:NCInt.self-eig} that 
 \begin{equation*}
  \frac1{\log N}\sum_{j<N} \lambda_j(A)=  \frac1{\log N}\sum_{j<N} (\lambda_j^+(A)-\lambda_j^-(A)) +\op{o}(1). 
\end{equation*}
This gives the result. 
\end{proof}

Let $\sH'$ be another Hilbert space.  
Theorem~\ref{thm:NCTint.Tauberian-measurable} implies the following spectral invariance result. 

\begin{proposition}\label{prop:NCInt.spectral-invariance}
 Let $A\in \sL_{1,\infty}(\sH)$ and $A'\in \sL_{1,\infty}(\sH')$ have the same non-zero eigenvalues with same multiplicities. Then $A$ is measurable if and only if $A'$ is measurable. Moreover, in this case $\bint A=\bint A'$. 
\end{proposition}

Suppose now that $\iota:\sH'\rightarrow \sH$ is continuous linear embedding, i.e., it is a linear map which is one-to-one and has closed range. For instance, any isometric linear map is such an embedding. Denote by $\sH_1$ the range of $\iota$. By assumption this is a closed subspace of $\sH$ and $\iota$ gives rise to a continuous linear isomorphism $\iota:\sH\rightarrow \sH_1$ with inverse $\iota^{-1}:\sH_1\rightarrow \sH'$. As explained in Appendix~\ref{sec:Appendix} we have a pushforward map $\iota_*: \sL(\sH')\rightarrow \sL(\sH)$ given by 
\begin{equation}
 \iota_*A=\iota \circ A\circ \iota^{-1} \circ \pi, \qquad A \in \sL(\sH'), 
 \label{eq:NCInt.iota*A}
\end{equation}
where $\pi:\sH \rightarrow \sH$ is the orthogonal projection onto $\sH_1$. In particular, if $\iota$ is invertible, then $\iota_*A=\iota A\iota^{-1}$.  We also know from Proposition~\ref{prop:App.embed-Schatten} that $\iota_*$ induces a continuous linear embedding,
\begin{equation*}
 \iota_*: \sL_{1,\infty}(\sH') \longrightarrow \sL_{1,\infty}(\sH).  
\end{equation*}
Moreover, by Proposition~\ref{prop:App.eigenvalues} if $A\in \sL_{1,\infty}(\sH')$, then $A$ and $\iota_*A$ have the same non-zero eigenvalues with same multiplicities. Combining this with Proposition~\ref{prop:NCInt.spectral-invariance} we then arrive at the following statement. 

\begin{corollary}\label{cor:NCInt.embedding} 
Let $A\in \sL_{1,\infty}(\sH')$. Then $\iota_*A$ is measurable if and only $A$ is measurable. Moreover, in this case we have
\begin{equation*}
 \bint \iota_*A = \bint A. 
\end{equation*}
\end{corollary}

Denote by $\sM(\sH)$ (resp., $\sM(\sH')$) the space of measurable operators on $\sH$ (resp., $\sH'$).
Specializing Corollary~\ref{cor:NCInt.embedding} to the case where $\iota$ is an isomorphism yields the following invariance result. 

\begin{corollary}\label{cor:NCInt.embedding-iso}
 Assume  $\iota:\sH\rightarrow \sH'$ is a continuous linear isomorphism. Then $\iota\sM(\sH)\iota^{-1}=\sM(\sH')$, and we have
\begin{equation*}
 \bint \iota A \iota^{-1} =\bint A \qquad \forall A\in \sM(\sH). 
\end{equation*}
\end{corollary}

\section{Connes' Integration and Lesbegue Integration}\label{sec:Lebesgue}  
 In this section, we look at the compatibility of Connes' integral with Lebesgue's integration. This will answer Question~B. 
 
 \subsection{Compatibility of Dixmier traces with Lebesgue's integration} 
To address the compatibility of Dixmier traces with Lebesgue's integration, the main technical hurdle is the lack of convexity of the weak trace class $\sL_{1,\infty}$. Indeed, $\sL_{1,\infty}$ is a quasi-Banach ideal, but this is not a Banach space or even a locally convex space. Thus, Bochner integration, or even Gel'fand-Pettis integration, of maps with values in $\sL_{1,\infty}$ do not make sense. We can remedy this by passing to the closure $\csL_{1,\infty}$ in the Dixmier-Macaev ideal $\sL_{(1,\infty)}$. Recall that
\begin{equation*}
 \sL_{(1,\infty)}:=\big\{A\in \sK; \ \sum_{j<N} \mu_j(T)=\op{O}(\log N)\big\}. 
\end{equation*}
 This is a Banach ideal with respect to the norm, 
\begin{equation}
 \|A\|_{(1,\infty)}= \sup_{N\geq 1} \frac{1}{\log(N+1)} \sum_{j<N}\mu_j(A),\qquad  A\in \sL_{(1,\infty)}.
 \label{eq:Int.Dixmier-Macaev-ideal}  
 \end{equation}
Note that $\csL_{1,\infty} \subsetneq \sL_{(1,\infty)}$ (see~\cite{KPS:AMS82, SS:JFA13}).

As $\csL_{1,\infty}$ equipped with the $\|\cdot\|_{(1,\infty)}$-norm is a Banach space, Bochner's integration makes sense for maps with values in $\csL_{1,\infty}$. Thus, given a measure space $(\Omega, \mu)$, for any measurable map $A:\Omega \rightarrow \sL_{1,\infty}$ we may at least define its Bochner integral $\int_\Omega A(x)d\mu(x)$ as an element of $\csL_{1,\infty}$ provided that $\int_\Omega \|A(x)\|_{(1,\infty)}d\mu(x)<\infty$. 

\begin{lemma}\label{lem:Int.extension-tau}
 The trace $\tau:\sL_{1,\infty}\rightarrow \ell_\infty/\ell_0$ uniquely extends to a positive linear trace $\overline{\tau}: \csL_{1,\infty}\rightarrow \ell_\infty/\ell_0$ which is continuous with respect to the Dixmier-Macaev norm~(\ref{eq:Int.Dixmier-Macaev-ideal}). 
\end{lemma}
\begin{proof}
 It follows from~(\ref{eq:NCInt.continuity-tau}) that there is $C>0$, such that, for all  $A\in \sL_{1,\infty}$, we have
\begin{equation*}
  \|\tau(A)\| \leq \sup_{N\geq 1} \frac1{\log N} \bigg|\sum_{j<N} \lambda_j(A)\bigg| \leq  \frac1{\log N}\sum_{j<N}\mu_j(A)\leq C\|A\|_{(1,\infty)}, 
\end{equation*}
Thus, the linear map $\tau$ is continuous with respect to the Dixmier-Macaev norm, and hence it uniquely extends to a continuous linear map $\overline{\tau}: \csL_{1,\infty}\rightarrow \ell_\infty/\ell_0$. This map is positive and is a trace. The proof is complete. 
\end{proof}

Given any state $\omega$ on $\ell_\infty/\ell_0$, we define the map $\bTrw:\csL_{1,\infty}\rightarrow \C$ by
\begin{equation*}
 \bTrw(A)=\omega\circ \overline{\tau}(A), \qquad A\in \csL_{1,\infty}. 
\end{equation*}
Equivalently, $\bTrw$ is the unique continuous extension to $\csL_{1,\infty}$ of the Dixmier trace $\Trw$.   

In what follows we let $(\Omega,\mu)$ be a measure space. 

\begin{proposition}\label{prop:Int.compatibility} 
 Let $A:\Omega \rightarrow \sL_{1,\infty}$ be a measurable map s.t.\ $\int_\Omega \|A(x)\|_{(1,\infty)}d\mu(x)<\infty$. Then, for every extended limit $\lim_\omega$, the function $\Omega\ni x\rightarrow \Trw[A(x)]$ is integrable, and we have
\begin{equation}
 \int_\Omega \Trw\big[A(x)\big]d\mu(x)= \bTrw \bigg( \int_\Omega A(x)d\mu(x)\bigg). 
 \label{eq:Int.swap} 
\end{equation}
In particular, if $\int_\Omega A(x)d\mu(x)\in \sL_{1,\infty}$, then 
\begin{equation*}
 \int_\Omega \Trw\big[A(x)\big]d\mu(x)= \Trw \bigg( \int_\Omega A(x)d\mu(x)\bigg). 
\end{equation*}
\end{proposition}
\begin{proof}
Let $\lim_\omega$ be an extended limit. The continuity of $\bTrw$ and the fact that $A:\Omega \rightarrow \sL_{1,\infty}$ is Bochner-integrable as an $\csL_{1,\infty}$-valued map ensure us that the function $\Trw[A(x)]=\bTrw[A(x)]$ is integrable, and we have
\begin{align*}
 \int_\Omega \Trw\big[A(x)\big]d\mu(x)  =  \int_\Omega  \bTrw\big[A(x)\big]d\mu(x)
 = \bTrw \bigg( \int_\Omega A(x)d\mu(x)\bigg). 
 \end{align*}
 The proof is complete.   
\end{proof}

\begin{corollary}
 Let $A:\Omega \rightarrow \sL_{1,\infty}$ be a measurable map  so that $\int_\Omega \|A(x)\|_{(1,\infty)}d\mu(x)<\infty$. Assume that $A(x)$ is a measurable operator a.e., and $\int_\Omega A(x)d\mu(x)\in \sL_{1,\infty}$. Then $\int_\Omega A(x)d\mu(x)$ is a measurable operator, and we have 
\begin{equation*}
 \int_\Omega \bigg(\bint A(x)\bigg)d\mu(x)= \bint \bigg( \int_\Omega A(x)d\mu(x)\bigg). 
\end{equation*}
\end{corollary}

\subsection{Dixmier traces associated with medial limits} 
As pointed out by Connes~\cite{Co:Survey19} another route to look at the compatibility of Connes' integration with Lebesgue integration is to use medial limits. These limits were introduced by Mokodoski~\cite{Me:SPS73}. Namely, by using the continuum hypothesis he proved the following result. 
 
\begin{lemma}[Mokodoski~\cite{Me:SPS73}] There exists a state $\omega:\ell_\infty/\ell_0\rightarrow \C$ which is universally measurable and such that, 
for any complete finite measure $\mu$ on $\ell_\infty/\ell_0$,  we have
\begin{equation}
 \omega\bigg(\int a d\mu(a)\bigg) = \int \omega(a)d\mu(a).
 \label{eq:Int.medial-property}  
\end{equation}
\end{lemma}

Let $\omega_\med$ be a state as in the above lemma. The corresponding extended limit is called a \emph{medial limit} and is denoted by $\limmed$. 

The fundamental property~(\ref{eq:Int.medial-property}) implies the following striking feature of medial limits. 

\begin{proposition}[\cite{Me:SPS73}] 
Given a complete finite measure space $(\Omega,\mu)$, let $(f_\ell)_{\ell\geq 1}$ a bounded family in $L^\infty(\Omega, \mu)$.
 \begin{enumerate}
 \item[(i)] The function $\Omega\ni x\rightarrow \limmed f_\ell(x)$ is measurable. 
 
 \item[(ii)] We have 
 \begin{equation*}
 \int_\omega \big( \limmed f_\ell(x) \big) d\mu(x)= \limmed \int_\Omega f_\ell(x)d\mu(x). 
\end{equation*}
\end{enumerate}
\end{proposition}

\begin{proposition}[\cite{Me:SPS73}] 
Given a complete finite measure space $(\Omega,\mu)$, let $(f_\ell)_{\ell\geq 1}$ a bounded family in $L^\infty(\Omega, \mu)$.
 Then $\Omega\ni x\rightarrow \limmed f_\ell(x)$ is a bounded measurable function such that
 \begin{equation*}
 \int_\omega \big( \limmed f_\ell(x) \big) d\mu(x)= \limmed \int_\Omega f_\ell(x)d\mu(x). 
\end{equation*}
\end{proposition}

In other words, we don't have to worry about the integrability of $\limmed f_\ell(x)$. We may freely swap the integral sign and the medial limit. Alternatively, if $a:\Omega \rightarrow \ell_\infty/\ell_0$ is any bounded measurable map, then
\begin{equation}
  \int_\Omega \omega_\med \big[ a(x) \big] d\mu(x)=\omega_\med\bigg( \int_\Omega a(x) d\mu(x)\bigg). 
  \label{eq:Int.key-property} 
\end{equation}

Denote by $\Trwmed$ the Dixmier trace associated with the extended limit $\limmed$. Let $A:\Omega \rightarrow \sL_{1,\infty}$ be a bounded measurable map. Applying~(\ref{eq:Int.key-property}) to $a(x)=\tau[A(x)]$ gives
\begin{equation*}
\int \omega_\med \circ \tau\big[A(x)\big]d\mu(x)  = \omega_\med\bigg( \int_\Omega \tau\big[A(x)] d\mu(x)\bigg). 
\end{equation*}
Note that $\omega_\med \circ \tau[A(x)]=\Trwmed[A(x)]$, and 
\begin{equation*}
 \int_\Omega \tau\big[A(x)] d\mu(x)=\int_\Omega \overline{\tau}\big[A(x)] d\mu(x)=\overline{\tau}
\bigg(\int_\Omega A(x) d\mu(x)\bigg). 
\end{equation*}
Thus, 
\begin{equation*}
\int \Trwmed\big[A(x)]d\mu(x)= \omega_\med \circ \overline{\tau}\bigg( \int_\Omega A(x) d\mu(x)\bigg) = \bTrwmed\bigg( \int_\Omega A(x) d\mu(x)\bigg). 
\end{equation*}
Therefore, in the special case of medial limits, we recover the formula~(\ref{eq:Int.swap}) as an immediate consequence of the fundamental property~(\ref{eq:Int.medial-property}) of those extended limits.

\section{Strongly Measurable Operators}\label{sec:strongly-measurable} 
In this section, we look at a stronger notion of measurability and show that we still have a sensible notion of integral on operators that satisfies this notion of measurability. 

\subsection{Strong measurability}
In what follows we denote by $T_0$ any positive operator in $\sL_{1,\infty}$ such that $\lambda_j(T_0)=(j+1)^{-1}$ for all $j\geq 0$. Note that any two such operators are unitary equivalent, and hence agree up to an element of $\Com(\sL_{1,\infty})$. 

Recall that a trace $\varphi:\sL_{1,\infty}\rightarrow \C$ is called \emph{normalized} if $\varphi(T_0)=1$. All the Dixmier traces are normalized traces. However, there are positive normalized traces on $\sL_{1,\infty}$ that are not Dixmier traces and don't have a continuous extension to the Dixmier-Macacev ideal $\sL_{(1,\infty)}$ (see, e.g.,~\cite[Theorem~4.7]{SSUZ:AIM15}). Therefore, it stands for reason to consider a stronger notion of measurability (see, e.g.,~\cite{KLPS:AIM13, LSZ:Book, SSUZ:AIM15}).  

\begin{definition}
An operator $T\in \sL_{1,\infty}$ is called \emph{strongly measurable} when there is $L\in \C$ such that $\varphi(T)=L$ for every positive normalized trace $\varphi$ on $\sL_{1,\infty}$. We denote by $\sMs$ the class of strongly measurable operators. 
 \end{definition}
\begin{remark}
 The class of strongly measurable operators is strictly contained in the space of measurable operators (see~\cite[Theorem~7.4]{SSUZ:AIM15}). 
\end{remark}

\begin{lemma}\label{lem:NCInt.positive-continuous-trace}
The space of continuous traces on $\sL_{1,\infty}$ is spanned by normalized positive traces. In fact, any continuous traces is a linear combination of 4 normalized positive traces.  
\end{lemma}
\begin{proof}
 Every positive trace on $\sL_{1,\infty}$ is continuous (see, e.g., \cite[Proposition~2.2]{Po:JMP20}). Conversely, every continuous trace on $\sL_{1,\infty}$  is a linear combination of 4 positive traces (see~\cite[Corollary~2.2]{CMSZ:ETDS19}). To complete the proof it is enough to show that every positive trace is a scalar multiple of a normalized positive trace. 
 
 Let  $\varphi$ be a non-zero positive trace. As $\sL_{1,\infty}$ is spanned by its positive cone, there is a positive operator $A\in \sL_{1,\infty}$ such that $\varphi(A)>0$. Let $(\xi_j)_{j\geq 0}$ be an orthonormal basis of $\sH$ such that $A\xi_j=\mu_j(T)\xi_j$ for all $j\geq 0$. Let $T_0$ be the operator on $\sH$ such that $T_0\xi_j=(j+1)^{-1} \xi_j$. As $\mu_j(A)\leq \|A\|_{1,\infty}(j+1)^{-1}$ we see that $A\leq \|A\|_{1,\infty}T_0$. The positivity of $\varphi$ then implies that $0<\varphi(A)\leq \|T\|_{1,\infty}\varphi(T_0)$. Thus, $\varphi(T_0)>0$, and so $\tilde{\varphi}:=\varphi(T_0)^{-1}\varphi$ is a  normalized positive trace. As $\varphi=\varphi(T_0)\tilde{\varphi}$ we see that every positive trace is a scalar multiple of a normalized positive trace. The proof is complete. 
\end{proof}

Lemma~\ref{lem:NCInt.positive-continuous-trace} implies the following characterization of strongly measurable operators in terms of continuous traces. 

\begin{lemma}\label{prop:NCInt.charact-strongly-measurable}
 Let $A\in \sL_{1,\infty}$. The following are equivalent:
\begin{enumerate}
 \item[(i)] $A$ is strongly measurable and $\bint A=L$. 
 
 \item[(ii)] $\varphi(A)=\varphi(T_0)L$ for every continuous trace on $\sL_{1,\infty}$. 
\end{enumerate}
\end{lemma}

This implies the following properties. 

\begin{proposition}\label{prop:NCInt.strong-measurable} 
 The following holds. 
\begin{enumerate}
 \item $\sMs$ is a closed subspace of $\sM$ containing $\Com(\sL_{1,\infty})$ and $(\sL_{1,\infty})_0$. In particular, every infinitesimal operator of order~$>1$ is strongly measurable. 
  
\item The space $\sMs$ does not depend on the inner product of $\sH$.   
  
 \item Let $A\in \sL_{1,\infty}$ be such that 
\begin{equation}
 \sum_{j<N} \lambda_j(A)= L \log N+ \op{O}(1).
 \label{eq:NCInt.super-Tauberian}
\end{equation}
Then $A$ is strongly measurable and $\bint A= L$. 
\end{enumerate}
\end{proposition}
\begin{proof}
 It is immediate that $\sMs$ is a subspace of $\sM$. Moreover, as the condition (ii) of Lemma~\ref{prop:NCInt.charact-strongly-measurable} depends only on the topological vector space structure of $\sL_{1,\infty}$, we see that $\sMs$ does not depend on the inner product of $\sH$. 
 
 Let $(A_\ell)_{\ell\geq 0}$ be sequence in $\sMs$ converging to $A$ in $\sL_{1,\infty}$. Set $\alpha_\ell=\bint A_\ell$. As $\sM$ is a closed subspace of 
 $\sL_{1,\infty}$ and $\bint$ is a continuous linear form, we see that $\alpha_\ell \rightarrow \alpha$ as $\ell \rightarrow \infty$. Let $\varphi$ be a positive normalized trace on $\sL_{1,\infty}$. We have $\varphi(A_\ell)=\bint A_\ell=\alpha_\ell$ for all $\ell \geq 0$. As $\varphi$ is a continuous trace, we have
\begin{equation*}
 \varphi(A)=\lim_{\ell \rightarrow \infty} \varphi(A_\ell)= \lim_{\ell \rightarrow \infty} \alpha_\ell=\alpha. 
\end{equation*}
Thus $\varphi(A)=\alpha$ for every positive normalized trace, and so $A$ is strongly measurable. This shows that $\sMs$ is a closed subspace of $\sM$. 

Furthermore, if $A\in \Com(\sL_{1,\infty})$ and $\varphi$ is a continuous trace on $\sL_{1,\infty}$, then $\varphi(A)=0$, and so by using Proposition~\ref{prop:NCInt.charact-strongly-measurable} we deduce that $A$ is strongly measurable. It follows from Proposition~\ref{prop:NCint.Com-DFWWW} that the ideal $\sR$ of finite rank operators is contained in $\Com(\sL_{1,\infty})$. As $\sMs$ is closed,  we deduce that it contains the closure of $\sR$ in $\sL_{1,\infty}$, i.e., the ideal $(\sL_{1,\infty})_0$.

Finally, let $A\in \sL_{1,\infty}$ satisfy~(\ref{eq:NCInt.super-Tauberian}). In addition, let $T_0$ be any positive operator in $\sL_{1,\infty}$ such that $\mu_j(T_0)=(j+1)^{-1}$. Then~(\ref{eq:NCInt.super-Tauberian}) is the r.h.s.~of~(\ref{eq:NCInt.A-B-Com}) for $B=LT_0$, and so $A-LT_0\in \Com(\sL_{1,\infty})$. In particular, given any normalized positive trace $\varphi$ on $\sL_{1,\infty}$ we have $\varphi(A)=L\varphi(T_0)=L$. That is,  $A\in \sMs$ and $\bint A=L$. The proof is complete. 
\end{proof}

\begin{remark}
 Let us called \emph{strongly Tauberian} any operator $A\in \sL_{1,\infty}$ satisfying~(\ref{eq:NCInt.super-Tauberian}). It same way as in the proof of Proposition~\ref{prop:NCInt.properties-Lambda}, it follows from Lemma~\ref{lem:NCInt.additivity} that the strongly Tauberian operators form a subspace of $\sMs$. This is a proper subspace. For instance, every operator in $\overline{\Com(\sL_{1,\infty})}\setminus \Com(\sL_{1,\infty})$ is strongly measurable, but is not strongly Tauberian. 
\end{remark}

\begin{remark}
As $\sMs$ is a subspace of $\sM$ containing $\Com(\sL_{1,\infty})$, we see that the NC integral $\bint$ induces a positive continuous trace on $\sMs$ which is annihilated by operators in $(\sL_{1,\infty})_0$. This restriction too satisfies the Ansatz for the NC integral.  
\end{remark}

\begin{remark}
 We refer to~\cite[Proposition 7.2]{SSUZ:AIM15} for a characterization of strongly measurable operators in terms of eigenvalue sequences. 
\end{remark}

In what follows we denote by $\sH'$ another Hilbert space. We also denote by $\sMs(\sH)$ (resp., $\sMs(\sH')$) the space of strongly measurable operators on $\sH$ (resp., $\sH'$). 
We have the following invariance property of strong measurable operators.

\begin{proposition}\label{prop:NCint.spectral-inv.strong}
 Let $A\in \sL_{1,\infty}(\sH)$ and $B\in \sL_{1,\infty}(\sH')$ have same non-zero eigenvalues with same multiplicities. Then 
 $A$ is strongly measurable if and only if $B$ is strongly measurable. 
\end{proposition}
\begin{proof}
By assumption any eigenvalue sequence for $A$ is an eigenvalue sequence for $B$, and vice versa. If $\sH'=\sH$, then combining this with Corollary~\ref{cor:NCint.Com} shows that $A-B$ is an operator in $\Com(\sL_{1,\infty})$, and hence is strongly measurable. As $\sMs(\sH)$ is a subspace of $\sL_{1,\infty}$, it follows that $A$ is strongly measurable if and only if $B$ is strongly measurable.

Suppose now that $\sH'\neq \sH$. Let $U:\sH\rightarrow \sH'$ be a unitary isomorphism, and set $B'=U^*BU$. Then $B'$ is an operator in $\sL_{1,\infty}(\sH)$ with same non-zero eigenvalues and same multiplicities as $A$ and $B$. Thus, by the first part of the proof, $A$ is strongly measurable if and only $B'$ is strongly measurable.

Let $\alpha_U:\sL(\sH')\rightarrow \sL(\sH)$ be defined by $\alpha_U(T)=U^*TU$, $T\in \sL(\sH)$. This is a $*$-isomorphism of $C^*$-algebras. It induces an isometric isomorphism $\alpha_U:\sL_{1,\infty}(\sH')\rightarrow \sL_{1,\infty}(\sH)$. By duality we get a one-to-one correspondance $\varphi\rightarrow \varphi\circ \alpha_U$ between positive traces on $\sL_{1,\infty}(\sH')$ and positive traces on $\sL_{1,\infty}(\sH)$. It then follows that $B$ is strongly measurable if and if only $B'$ is strongly measurable. This gives the result when $\sH'\neq \sH$. The proof is complete. 
\end{proof}

Using the same the same type of argument  that lead to Corollary~\ref{cor:NCInt.embedding-iso} and Corollary~\ref{cor:NCInt.embedding} we obtain the following consequence of Proposition~\ref{prop:NCint.spectral-inv.strong}. 

\begin{corollary}
 Let $\iota:\sH'\rightarrow \sH$ be a Hilbert space embedding.
\begin{enumerate}
 \item If $A\in \sL_{1,\infty}(\sH')$, then $\iota_*A$ is strongly measurable if and if only $A$ is strongly measurable. 
 
 \item If $\iota$ is an isomorphism, then $\iota \sMs(\sH') \iota^{-1}=\sMs(\sH)$. 
\end{enumerate}
\end{corollary}

\subsection{Connes' trace theorem}\label{subsec.PsiDOs} 
Suppose that $(M^n,g)$ is a closed Riemannian manifold and $E$ is a Hermitian vector bundle. Given $m\in \R$,  we denote by $\Psi^m(M,E)$ the space of $m$-th order classical pseudodifferential operators (\psidos) $P:C^\infty(M,E)\rightarrow C^\infty(M,E)$. If $P\in \Psi^m(M,E)$, then we denote by $\sigma(P)(x,\xi)$ its principal symbol; this is a smooth section of $\End(E)$ over $T^*M\setminus 0$. Any $P\in \Psi^m(M,E)$ with $m\leq 0$ extends to a bounded operator $P:L_2(M,E)\rightarrow L_2(M,E)$. If in addition $m<0$, then we get an operator in the weak Schatten class $\sL_{p,\infty}$ with $p=n|m|^{-1}$, i.e., an infinitesimal operator of order~$1/p$. 

Setting $\Psi^\Z(M,E)=\bigcup_{m\in \Z}\Psi^m(M,E)$, let $\Res :\Psi^\Z(M,E)\rightarrow \C$ be the noncommutative residue trace of Guillemin~\cite{Gu:AIM85} and Wodzicki~\cite{Wo:NCRF}. It appears as the residual trace on integer-order \psidos\ induced by the analytic extension of the ordinary trace to all non-integer order \psidos~(see~\cite{Gu:AIM85, Wo:NCRF}). A result of Wodzicki~\cite{Wo:HDR} further asserts this is the unique trace up to constant multiple on the algebra $\Psi^\Z(M,E)$ if $M$ is connected and has dimension $n\geq 2$ (see also~\cite{Le:AGAG99, Po:JAM10}). The noncommutative residue is a local functional. Namely, if $P\in \Psi^\Z(M,E)$, then
\begin{equation*}
 \Res(P) = \int_M \tr_E[c_P(x)], 
\end{equation*}
where $c_P(x)$ is an $\End(E)$-valued 1-density which is given in local coordinates by
\begin{equation*}
 c_P(x)=(2\pi)^{-n}\int_{|\xi|=1} a_{-n}(x,\xi) d^{n-1}\xi,
\end{equation*}
 where $a_{-n}(x,\xi)$ is the symbol of degree $-n$ of $P$. In particular, if $P$ has order~$-n$, then 
 \begin{equation*}
 \Res(P)= (2\pi)^{-n} \int_{S^*M} \tr_E\big[\sigma(P)(x,\xi)\big] dxd\xi,
\end{equation*}
where $S^*M=T^*M/\R^*_+$ is the cosphere bundle and $dxd\xi$ is the Liouville measure. 

\begin{proposition}[Connes's Trace Theorem~\cite{Co:CMP88, KLPS:AIM13}] \label{prop:NCInt.Connes-trace-thm} 
Every operator $P\in \Psi^{-n}(M,E)$ is strongly measurable, and we have
\begin{equation}
 \bint P= \frac{1}{n} \Res(P). 
 \label{eq:Connes-trace-thm} 
\end{equation}
 \end{proposition}

\begin{remark}
 Connes~\cite{Co:CMP88} established measurablity and derived the trace formula~\ref{eq:Connes-trace-thm}.  
 Kalton-Lord-Potapov-Sukochev~\cite{KLPS:AIM13} obtained strong measurability.
 We observe that Connes' arguments can also be used to get strong measurability. 
\end{remark}

Suppose that $E$ is the trivial line bundle, and let $\Delta_g$ be the Laplace-Beltrami operator on functions. As an application of Connes' trace theorem we obtain the following integration formula, which shows that the noncommutative integral recaptures the Riemannian volume density. 

\begin{proposition}[Connes's Integration Formula~\cite{Co:CMP88, KLPS:AIM13}]\label{prop:NCInt.Connes-int-formula} 
 For all $f\in C^\infty(M)$, the operator $f \Delta_g^{-\frac{n}{2}}$ is strongly measurable, and we have
\begin{equation}
 \bint \Delta_g^{-\frac{n}{4}} f \Delta_g^{-\frac{n}{4}} = c_n \int_M f(x) \sqrt{\det\left(g(x)\right)}d^nx, \qquad c_n:=\frac{1}{n} (2\pi)^{-n}|\bS^{n-1}|. 
 \label{eq:NCInt.Connes-int-formula}
\end{equation}
\end{proposition}
\begin{remark}\label{rmk:Strong.Orlicz} 
 The integration formula~(\ref{eq:NCInt.Connes-int-formula}) fails in general for functions in $L_1(M)$ (see~\cite{KLPS:AIM13}). However, as shown by 
 Rozenblum~\cite{Ro:arXiv21} and Sukochev-Zanin~\cite{SZ:arXiv21} (see also~\cite{Po:Weyl-Orlicz}) it actually holds for any function in the Orlicz space $\LlogL(M)$, i.e., measurable functions $f$ on $M$ such that $\int (1+|f|)\log(1+|f|)\sqrt{g} dx<\infty$. In particular, it holds for any $f\in L_p(M)$ with $p>1$ (see also~\cite{KLPS:AIM13, LPS:JFA10, LSZ:JFA20}). In fact, Rozenblum~\cite{Ro:arXiv21} further extends this result to potentials of the form $f=h\mu$, where $h\in \LlogL(M)$ and $\mu$ is in a suitable class of Borel measures (see also~\cite{RS:EMS21}). 
 \end{remark}

\begin{remark}
We refer to~\cite{Po:JFA07} for versions of Connes' trace theorem and Conne's integration formulas for Heisenberg pseudodifferential operators on contact manifolds and Cauchy-Riemann manifolds. We also refer to~\cite{MP:Part1, MP:Part2, Po:JMP20} for extensions of Connes' trace theorem and Connes's integration formula to noncommutative tori. In addition,  versions of Connes' integration formula for noncommutative Euclidean spaces and $\op{SU}(2)$ are given in~\cite{MSZ:MA19}.
\end{remark}

\section{Weyl's Laws and Noncommutative Integration}\label{sec:Weyl} 
In this section, we relate Connes' integration to the Weyl's laws for compact operators studied by Birman-Solomyak~\cite{BS:JFAA70} and others in the late 60s and early 70s. In particular, this will exhibit an even stronger notion of measurability of purely spectral nature, and so this will provide another spectral theoretic interpretation of Connes' integral. 

\subsection{Weyl operators} 
If $A$ is a selfadjoint compact operator, then as in Remark~\ref{rmk:NCInt.self-eig} we denote by $(\pm\lambda^\pm(A))_{j\geq 0}$ its sequences of positive and negative eigenvalues, i.e., $\lambda_j^\pm(A)=\lambda_j(A^\pm)=\mu_j(A^{\pm})$, where $A^\pm=\frac12(|A|\pm A)$ are the positive and negative parts of $A$. 
 We refer to~\cite[\S9.2]{BS:Book} for the main properties of the positive/negative eigenvalue sequences of selfadjoint compact operators. In particular, we have the following min-max principle (\emph{cf}.~\cite[Theorem 9.2.4]{BS:Book}), 
\begin{equation*}
 \lambda^\pm_j(A)= \min \bigg\{ \max_{0\neq \xi\in E^\perp} \pm \frac{\scal{A\xi}{\xi}}{\scal\xi\xi}; \  \dim E=j\bigg\}, \qquad j\geq 0. 
\end{equation*}
This implies the following version of Ky Fan's inequality (\emph{cf}.~\cite[Theorem 9.2.8]{BS:Book}), 
\begin{equation}
 \lambda^\pm_{j+k}(A+B)\leq \lambda^\pm_j(A) + \lambda^\pm_k(B), \qquad j,k\geq 0. 
 \label{eq:Bir-Sol.Weyl-ineq-lambdapm}
\end{equation}

\begin{definition}
We say that $A\in \sL_{p,\infty}$, $p>0$, is a \emph{Weyl operator} if one of the following conditions applies:
\begin{enumerate}
 \item[(i)] $A\geq 0$ and $\lim j^{1/p}\lambda_j(A)$ exists. 
 
 \item[(ii)] $A^*=A$ and $\lim j^{1/p}\lambda_j^+(A)$ and $\lim j^{1/p}\lambda_j^-(A)$ both exist. 
 
 \item[(iii)] The real part $\Re A=\frac12(A+A^*)$ and the imaginary part $\Im A= \frac1{2i}(A-A^*)$ of are both Weyl operators in the sense of (ii). 
\end{enumerate}
\end{definition}

We denote by $\sW_{p,\infty}$ the class of Weyl operators in $\sL_{p,\infty}$.  If $A\in \sW_{p,\infty}$, $A\geq 0$, we set 
\begin{equation*}
 \Lambda(A)=\lim_{j\rightarrow \infty} j^{\frac1p} \lambda_j(A). 
\end{equation*}
If $A=A^*\in \sW_{p,\infty}$, we set 
\begin{equation*}
  \Lambda^\pm(A)=\lim_{j\rightarrow \infty} j^{\frac1p} \lambda^\pm_j(A). 
\end{equation*}
For an arbitrary operator $A\in \sW_{p,\infty}$, we define 
\begin{equation*}
  \Lambda^\pm(A)=  \Lambda^\pm\big(\Re A\big) + i \Lambda^\pm\big(\Im A\big). 
\end{equation*}

\begin{remark}
 If $A=A^*\in \sL_{p,\infty}$, then
\begin{equation}
0\leq \lambda_j^\pm(A)\leq \mu_j(A) \leq (j+1)^{-\frac1{p}}\|A\|_{p,\infty} \qquad \forall j\geq 0. 
\label{eq:Bir-Sol.lambdpm-mu}
\end{equation}
 In particular, the sequences $(j^{1/p}\lambda^{\pm}(A))_{j\geq 0}$ are always bounded. If in addition $A$ is a Weyl operator, then we have
\begin{equation}
 0\leq \Lambda^\pm(A) \leq \|A\|_{p,\infty}. 
 \label{eq:Bir-Sol.Lambda-norm}
\end{equation}
\end{remark}

\begin{remark}\label{rmk:Bir-Sol.lunf0} 
 If $A=A^*\in (\sL_{p,\infty})_0$, then $j^{1/p}\mu_j(A)\rightarrow 0$, and so by using~(\ref{eq:Bir-Sol.lambdpm-mu}) we see that $j^{1/p}\lambda_j^\pm(A)\rightarrow 0$ as well. Thus, $A\in \sW_{p,\infty}$, and $\Lambda^\pm(A)=0$. More generally, by taking real and imaginary parts we see that every operator $A\in (\sL_{p,\infty})_0$ is contained in $\sW_{p,\infty}$ with $\Lambda^\pm(A)=0$. This includes all infinitesimal operators of order~$>p$. 
\end{remark}

\begin{remark}\label{rmk:Bir-Sol.counting}
 Given any selfadjoint compact operator $A$ on $\sH$, its counting functions are given by
\begin{equation*}
 N^\pm(A;\lambda):=\# \big\{j; \ \lambda_j^\pm(A)>\lambda\big\}, \qquad \lambda>0. 
\end{equation*}
If $A\in \sW_{p,\infty}$, $p>0$, then (see, e.g., \cite[Proposition~13.1]{Sh:Springer01}), we have
\begin{equation}
\lim_{\lambda \rightarrow 0^+} \lambda^pN^\pm(A;\lambda)= \lim_{j\rightarrow \infty} j (\lambda_j^\pm(A))^p= \Lambda^\pm(A)^p. 
\label{eq:Bir-Sol.counting-Lambda}
\end{equation}
\end{remark}

In addition, it will be convenient to introduce the following class of operators. 

\begin{definition}
 $\sW_{|p,\infty|}$, $p>0$, consists of operators $A\in \sL_{p,\infty}$ such that $|A|\in \sW_{p,\infty}$, i.e., $\lim j^{1/p}\mu_j(A)$ exists. 
\end{definition}

In particular, if $A\in \sW_{|p,\infty|}$, then 
\begin{equation*}
 \Lambda\big(|A|\big)= \lim_{j\rightarrow \infty} j^{\frac1p}\mu_j(A). 
\end{equation*}

\subsection{Birman-Solomyak's perturbation theory} 
We recall the main facts regarding the perturbation theory of Birman-Solomyak~\cite[\S4]{BS:JFAA70}. 

\begin{proposition}[Birman-Solomyak~{\cite[Theorem 4.1]{BS:JFAA70}}] \label{prop:Bir-Sol.Perturbation}
Let $A=A^*\in \sL_{p,\infty}$. Assume that, for every $\epsilon>0$, we may write 
\begin{equation*}
 A=A_\epsilon'+A_\epsilon'', 
\end{equation*}
where $A'_\epsilon$ and $A''_\epsilon$ are selfadjoint operators in $\sL_{p,\infty}$ such that $A'_\epsilon \in \sW_{p,\infty}$, and 
\begin{equation*}
 \limsup j^{\frac1p} \lambda^\pm_j(A''_\epsilon)\leq \epsilon. 
\end{equation*}
 Then $A\in \sW_{p,\infty}$, and we have
\begin{equation*}
 \lim_{\epsilon \rightarrow 0^+} \Lambda^\pm\big(A'_\epsilon\big)= \Lambda^{\pm}(A). 
\end{equation*}
\end{proposition}

We stress that Proposition~\ref{prop:Bir-Sol.Perturbation} is obtained in~\cite{BS:JFAA70} as a sole consequence of the Ky Fan's inequality~(\ref{eq:Bir-Sol.Weyl-ineq-lambdapm}). This result has a number of consequences. 

\begin{corollary}\label{cor:Bir-Sol.closedness} 
 $\sW_{p,\infty}$ is a closed subset of $\sL_{p,\infty}$ on which $\Lambda^\pm:\sW_{p,\infty}\rightarrow \C$ are continuous maps. 
\end{corollary}
\begin{proof}
 We only need to show that if $(A_\ell)_{\ell\geq 0}$ is a sequence in $\sW_{p,\infty}$ converging to $A$ in $\sL_{p,\infty}$, then $A\in \sW_{p,\infty}$, and we have 
\begin{equation}
 \lim_{\ell \rightarrow \infty} \Lambda^\pm(A_\ell)=\Lambda^\pm(A).
 \label{eq:Bir-Sol.lim-Lambdapm}  
\end{equation}
By taking real and imaginary parts we may assume that the operators $A_\ell$ and $A$ are selfadjoint. 

Thanks to~(\ref{eq:Bir-Sol.Lambda-norm}) we have $\sup_{\ell \geq 0} \Lambda^{\pm}(A_\ell)\leq \sup_{\ell \geq 0}\|A_\ell\|_{p,\infty} <\infty$. That is, $\{\Lambda^{\pm}(A_\ell)\}_{\ell\geq 0}$ are bounded sequences. Let $\{\Lambda^{\pm}(A_{\ell_p})\}_{p\geq 0}$ be convergent subsequences. As $A_{\ell_p}\rightarrow A$ in $\sL_{p,\infty}$, given any $\epsilon >0$, there is $p_{\epsilon}>\epsilon^{-1}$  such that $\|A-A_{\ell_{p_{\epsilon}}}\|<\epsilon$. In view of~(\ref{eq:Bir-Sol.lambdpm-mu}) this implies that 
\begin{equation*}
 \limsup j^{\frac1p} \lambda_j^\pm\big(A-A_{\ell_{p_{\epsilon}}}\big) \leq \big\|A-A_{\ell_{p_{\epsilon}}}\big\|<\epsilon.
\end{equation*}
As $A_{\ell_{p_{\epsilon}}}\in \sW_{p,\infty}$ for all $\epsilon >0$, it follows from Proposition~\ref{prop:Bir-Sol.Perturbation} that $A\in \sW_{p,\infty}$ and we have
\begin{equation*}
 \Lambda(A)=\lim_{\epsilon \rightarrow 0^+} \Lambda^\pm\big(A_{\ell_{p_{\epsilon}}}\big)=\lim_{p\rightarrow \infty}  \Lambda^\pm\big(A_{\ell_p}\big).
\end{equation*}
This shows that $\Lambda^\pm(A)$ are the unique limit points of the bounded sequences  $\{\Lambda^{\pm}(A_\ell)\}_{\ell\geq 0}$. This gives~(\ref{eq:Bir-Sol.lim-Lambdapm}). The proof is complete. 
\end{proof}

\begin{corollary}[K.~Fan~{\cite[Theorem 3]{Fa:PNAS51}}; see also~{\cite[Theorem II.2.3]{GK:AMS69}}] \label{cor:Bir-Sol.lunfo-perturbation} 
 Let $A\in \sW_{p,\infty}$ and $B\in (\sL_{p,\infty})_0$. Then $A+B\in \sW_{p,\infty}$, and we have 
\begin{equation*}
  \Lambda^\pm(A+B)=\Lambda^\pm(A). 
\end{equation*}
\end{corollary}
\begin{proof}
 It is enough to prove the result when $A$ and $B$ are selfadjoint. In this case we know from Remark~\ref{rmk:Bir-Sol.lunf0} that $\lim j^{1/p}\lambda_j^\pm(B)=0$. Therefore, the 
we get the result by applying Proposition~\ref{prop:Bir-Sol.Perturbation} to $A+B$ with $A'_\epsilon=A$ and $A''_\epsilon=B$ for all $\epsilon>0$. 
\end{proof}

As mentioned above Proposition~\ref{prop:Bir-Sol.Perturbation} is a sole consequence of the Ky Fan's inequality~(\ref{eq:Bir-Sol.Weyl-ineq-lambdapm}). As the singular values satisfy the Ky Fan's inequality~(\ref{eq:Quantized.properties-mun2}), we similarly have a version of Proposition~\ref{prop:Bir-Sol.Perturbation} for singular values (\emph{cf}.~\cite[Remark 4.2]{BS:JFAA70}). 
Namely, we have the following result. 

\begin{proposition}[Birman-Solomyak~\cite{BS:JFAA70}] Let $A\in \sL_{p,\infty}$. Assume that, for every $\epsilon>0$, we may write 
\begin{equation*}
 A=A_\epsilon'+A_\epsilon'', 
\end{equation*}
with $A'_\epsilon \in \sW_{|p,\infty|}$ and $A''_\epsilon \in \sL_{p,\infty}$ such that
\begin{equation*}
 \limsup j^{\frac1p} \mu_j(A''_\epsilon)\leq \epsilon. 
\end{equation*}
 Then $A\in \sW_{|p,\infty|}$, and we have
\begin{equation*}
 \lim_{\epsilon \rightarrow 0^+} \Lambda\big(|A'_\epsilon|\big)= \Lambda(|A|). 
\end{equation*}
\end{proposition}

By arguing along the same lines as that of the proofs of Corollary~\ref{cor:Bir-Sol.closedness} and Corollary~\ref{cor:Bir-Sol.lunfo-perturbation} we arrive at the following statements. 

\begin{corollary}\label{cor:Bir-Sol.closedness-sing}
$ \sW_{|p,\infty|}$ is a closed subset of $\sL_{p,\infty}$ on which the functional $A \rightarrow \Lambda(|A|)$ is continuous. 
\end{corollary}

\begin{corollary}[K.~Fan~{\cite[Theorem 3]{Fa:PNAS51}}; see also~{\cite[Theorem II.2.3]{GK:AMS69}}] \label{cor:Bir-Sol.lunfo-perturbation-sing}
 If $A\in \sW_{|p,\infty|}$ and $B\in (\sL_{p,\infty})_0$. Then $A+B\in \sW_{|p,\infty|}$, and we have 
\begin{equation*}
  \Lambda^\pm\big(|A+B|\big)=\Lambda^\pm\big(|A|\big). 
\end{equation*}
\end{corollary}

\subsection{Measurability of Weyl Operators} 
Suppose that $p=1$. We shall now show that every Weyl operator in $\sL_{1,\infty}$ is strongly measurable and explain how to compute its NC integral in terms of the maps $\Lambda^\pm$. More precisely, we shall prove the following result. 

\begin{proposition}\label{prop:Bir-Sol.Weyl-mesurable} 
 Let $A\in \sW_{1,\infty}$. Then $A$ is strongly measurable, and we have
\begin{equation}
 \bint A = \Lambda^+(A)- \Lambda^-(A). 
 \label{eq:Bir-Sol.bint-Lambdapm}
\end{equation}
In particular, if $A$ is selfadjoint, then 
\begin{equation*}
 \bint A = \lim_{j\rightarrow \infty}j\lambda_j^+(A) -  \lim_{j\rightarrow \infty}j\lambda_j^-(A). 
\end{equation*}
\end{proposition}
\begin{proof}
 Let us first  show that $A$ is measurable and its integral is given by~(\ref{eq:Bir-Sol.bint-Lambdapm}). By taking real and imaginary parts and using Proposition~\ref{prop:NCInt.Measurable-Re-Im} we may assume that $A$ is selfadjoint. In this case we have 
\begin{equation*}
 \lim_{N\rightarrow \infty} \frac1{\log N}\sum_{j<N} (\lambda_j^+(A)-\lambda_j^-(A)) = 
  \lim_{j\rightarrow \infty}j\lambda_j^+(A) -  \lim_{j\rightarrow \infty}j\lambda_j^-(A) =\Lambda^+(A)- \Lambda^-(A). 
\end{equation*}
Combining this with Proposition~\ref{prop:NCInt.Measurable-selfadj} shows that $A$ is measurable and $\bint A=\Lambda^+(A)- \Lambda^-(A)$. 
 
 It remains to show that any $A\in \sW_{1,\infty}$ is strongly measurable. By taking real and imaginary and  their respective positive and negative parts reduces to the case $A\geq 0$, which we assume thereon. Let $(\xi_j)_{j\geq 0}$ be an orthonormal basis of $\sH$ such that $A\xi_j=\lambda_j(A)$, and let $T_0$ be the operator on $\sH$ such that $T_0\xi_j=(j+1)^{-1}\xi_j$ for all $j\geq 0$. Note that $T_0$ is strongly measurable. 
 
Set $B=A-\Lambda(A)T_0$.  For all $j \geq 0$, we have
 \begin{equation*}
 B\xi_j=\gamma_j\xi_j, \qquad \text{where}\ \gamma_j:=\lambda_j(A)-\Lambda(A)(j+1)^{-1}. 
\end{equation*}
Moreover, the fact that $j\lambda_j(A)\rightarrow \Lambda(A)$ implies that  $j\gamma_j\rightarrow 0$ as $j\rightarrow \infty$. 

Let $j\geq 0$. By applying the min-max principle~(\ref{eq:min-max}) and taking $E=\op{Span}\{\xi_k; k<j\}$ we get
\begin{equation*}
 \mu_j(B) \leq \big\|B_{|E^\perp}\big\| = \sup_{k\geq j}|\gamma_k|. 
\end{equation*}
This implies that $ j\mu_j(B) \leq \sup_{k\geq j}j|\gamma_k| \leq  \sup_{k\geq j}k|\gamma_k|$.  Thus,
\begin{equation*}
 \limsup_{j\rightarrow \infty} j\mu_j(B) \leq \lim_{j\rightarrow \infty}  \sup_{k\geq j}k|\gamma_k|= \limsup_{j\rightarrow \infty} j|\gamma_j|= \lim_{j\rightarrow \infty} j|\gamma_j|=0. 
\end{equation*}
This shows that $B$ is in $ (\sL_{1,\infty})_0$, and hence is strongly measurable by Proposition~\ref{prop:NCInt.strong-measurable}. As $A=B+\Lambda(T)T_0$, it follows that $A$ is strongly measurable as well. The proof is complete. 
\end{proof}

\begin{corollary}\label{cor:Bir-Sol.Weyl-mesurable-|A|}
 Let $A\in  \sW_{|1,\infty|}$. Then $|A|$ is strongly measurable, and we have
\begin{equation*}
  \bint |A| = \lim_{j\rightarrow \infty}j\mu_j(A)
\end{equation*}
\end{corollary}

\section{Weyl's Laws for Negative Order \psidos} \label{sec:Weyl-neg-PsiDOs}\label{sec:Bir-Sol-PDOs} 
In the 70s Birman-Solomyak~\cite{BS:VLU77, BS:VLU79, BS:SMJ79} established a Weyl's law for negative order \psidos. Unfortunately, the main key technical details are exposed in a somewhat compressed manner in the Russian article~\cite{BS:VLU79}, the translation of which remains unavailable. 

In this section, after explaining how this implies a stronger version of Connes's trace theorem, we shall provide a ``soft proof'' of Birman-Solomyak's result. This will answer Question~D. 

\subsection{Weyl's law for negative order \psidos} 
In the following $(M^n,g)$ is a closed Riemannian manifold and $E$ is a Hermitian vector bundle over $M$. We keep on using the notation of \S\S\ref{subsec.PsiDOs}. 

\begin{theorem}[Birman-Solomyak~\cite{BS:VLU77, BS:VLU79, BS:SMJ79}]\label{thm:Bir-Sol-asymp} Let $P\in \Psi^{-m}(M,E)$, $m<0$, and set $p=nm^{-1}$. 
\begin{enumerate}
 \item $P$ and $|P|$ are Weyl operators in $\sL_{p,\infty}$. 
 
 \item  We have
\begin{equation*}
 \lim_{j\rightarrow \infty} j^{\frac1p} \mu_j(P)=\bigg[ \frac1{n} (2\pi)^{-n} \int_{S^*M} \tr_E\big[ |\sigma(P)(x,\xi)|^{p} \big] dx d\xi\bigg]^{\frac1{p}}.
 \label{eq:Weyl.Bir-Sol-mu} 
\end{equation*}

 \item If $P$ is selfadjoint, then 
\begin{equation*}
 \lim_{j\rightarrow \infty} j^{\frac1p} \lambda^\pm_j(P)= \bigg[\frac1{n} (2\pi)^{-n} \int_{S^*M} \tr_E\big[ \sigma(P)(x,\xi)_\pm^{p} \big] dx d\xi\bigg]^{\frac1{p}}. 
 \label{eq:Weyl.Bir-Sol-selfadjoint}
\end{equation*}
\end{enumerate}
 \end{theorem}

\begin{remark}\label{rmk.Bir-Sol.non-smooth} 
In~\cite{BS:VLU77, BS:VLU79} Birman-Solomyak established the Weyl's laws above for compactly supported pseudodifferential operators on $\R^n$ under very low regularity assumptions on the symbols. Furthermore, the symbols are allowed to be anisotropic. This was extended to classical \psidos\ on closed manifolds in~\cite{BS:SMJ79}. 
\end{remark}

\begin{remark}
 We refer to~\cite{AA:FAA96, An:MUSSRS90, BY:JSM84, DR:LNM87, Gr:CPDE14, Iv:Springer19, Ro:arXiv21, RS:EMS21}, and the references therein, for various generalizations and applications of Birman-Solomyak's asymptotics.  
\end{remark}

Combining Theorem~\ref{thm:Bir-Sol-asymp} with Corollary~\ref{prop:Bir-Sol.Weyl-mesurable} and Corollary~\ref{cor:Bir-Sol.Weyl-mesurable-|A|} provides us with a stronger form of Connes' trace theorem (compare Proposition~\ref{prop:NCInt.Connes-trace-thm}). 

\begin{theorem}\label{thm:Weyl.strong-CTT} 
 Let $P\in \Psi^{-n}(M,E)$. The following holds. 
\begin{enumerate}
 \item $P$ and $|P|$ are Weyl operators in $\sL_{1,\infty}$, and hence are strongly measurable. 
 
 \item We have 
\begin{gather}
 \bint P = \Lambda(P)=  \frac1{n} (2\pi)^{-n} \int_{S^*M} \tr_E\big[ \sigma(P)(x,\xi) \big] dx d\xi ,\\
 \bint |P|= \lim_{j\rightarrow \infty} j\mu_j(P)= \frac1{n} (2\pi)^{-n} \int_{S^*M} \tr_E\big[ |\sigma(P)(x,\xi)| \big] dx d\xi.
\label{eq:Weyl.Bir-Sol-mu-lunf}  
\end{gather}

\item If $P$ is selfadjoint, then 
\begin{equation}
 \lim_{j\rightarrow \infty} j\lambda^\pm_j(P)= \frac1{n} (2\pi)^{-n} \int_{S^*M} \tr_E\big[ \sigma(P)(x,\xi)_\pm \big] dx d\xi.
 \label{eq:Weyl.Bir-Sol-selfad-lunf}   
\end{equation}
\end{enumerate}
\end{theorem}

In what follows we let $\Delta_E=\nabla^*\nabla$ be the Laplacian of some Hermitian connection $\nabla$ on $E$. In particular, $\Delta_E$ is (formally) selfadjoint and has principal symbol $\sigma(\Delta_E)=|\xi|^2\op{id}_{E_x}$ (where we denote by $|\cdot|$ the Riemannian metric on $T^*M$). 

Theorem~\ref{thm:Bir-Sol-asymp} leads to the following improvement of Connes' integration formula~(\ref{eq:NCInt.Connes-int-formula}).

\begin{corollary}\label{cor:Weyl.strong-CIF-smooth} 
 Let $u\in C^\infty(M,\End (E))$. The following hold. 
 \begin{enumerate}
\item The operators $(\Delta^E)^{-n/4}u(\Delta^E)^{-n/4}$ and $|(\Delta^E)^{-n/4}u(\Delta^E)^{-n/4}|$ are Weyl operators in $\sL_{1,\infty}$, and hence are strongly measurable. 

\item We have 
 \begin{gather}
 \bint \Delta_E^{-\frac{n}{4}}u\Delta_E^{-\frac{n}{4}} = \frac1{n} (2\pi)^{-n} \ \int_M \tr_E\big[u(x)\big] \sqrt{g(x)}dx,
 \label{eq:Weyl.Integration-formula}\\
 \bint  \left|\Delta_E^{-\frac{n}{4}}u\Delta_E^{-\frac{n}{4}}\right|= \lim_{j\rightarrow \infty} \mu_j \left(\Delta_E^{-\frac{n}{4}}u\Delta_E^{-\frac{n}{4}}\right)=
  \frac1{n} (2\pi)^{-n} \ \int_M \tr_E\big[|u(x)|\big] \sqrt{g(x)}dx. 
\end{gather}

\item If $u(x)^*=u(x)$, then 
\begin{equation}
  \lim_{j\rightarrow \infty} j\lambda^\pm_j\left(\Delta_E^{-\frac{n}{4}}u\Delta_E^{-\frac{n}{4}}\right)= \frac1{n} (2\pi)^{-n}  \int_M  \tr_E\big[u(x)_\pm\big] \sqrt{g(x)}dx.
  \label{eq:Weyl.Weyllaw} 
\end{equation}
\end{enumerate}
\end{corollary}

\begin{remark}
The above results continue to hold for operators of the form $QuP$, where $P$ and $Q$ are operators in $\Psi^{-n/2}(M,E)$ and $u(x)$ is a potential in the Orlicz class $\LlogL(M, \End(E))$ (for the 3rd part we take $Q=P^*$) (see~\cite{Ro:arXiv21, Po:Weyl-Orlicz}; see also~\cite{SZ:arXiv21}). In particular, the operators $P$ and $Q$ need not be a negative power of an elliptic operator. Rozenblum~\cite{Ro:arXiv21} actually establishes the results in the scalar case for potentials of the form $u=h\mu$, where $h$ is in the Orlicz space $\LlogL(M)$ and $\mu$ is in a suitable class of Borel measures (see also~\cite{RS:EMS21}). 
If in addition, if $E$ is a Clifford module, then we may replace $\Delta_E$ by $\ssD_E^2$, where $\ssD_E$ is the Dirac operator associated to some unitary Clifford connection on $E$.
\end{remark}

\subsection{Proof of Theorem~\ref{thm:Bir-Sol-asymp}}
We will deduce Theorem~\ref{thm:Bir-Sol-asymp} from the properties of zeta functions of elliptic operators and their relationship with the noncommutative residue trace. This will clarify the relationship between Birman-Solomyak's result and the noncommutative residue. More precisely, we shall use the following result. 

\begin{proposition}[\cite{Gu:AIM85, Wo:NCRF}]\label{prop:Weyl.NCR-zeta} 
 Let $Q\in \Psi^{m}(M,E)$, $m>0$, be elliptic and let $A\in \Psi^0(M,E)$. The function $z\rightarrow \Tr[A|Q|^{-z}]$ has a meromorphic extension to $\C$ with at worst simple pole singularities on $\Sigma:=\{km^{-1}; \ k\in \Z, \ k\leq n\}$ in such way that
\begin{equation*}
 \Res_{z=\sigma}  \Tr\big[A|Q|^{-z}\big]= \frac{1}{m} \Res\big(A|Q|^{-\sigma}\big), \qquad \sigma \in \Sigma. 
\end{equation*}
\end{proposition}

As is well known, the above result implies the following Weyl's laws. 

\begin{corollary}\label{cor:Weyl.Weyl-invere-elliptic} 
 Let $Q\in \Psi^{m}(M,E)$, $m>0$, be elliptic. Set $p=nm^{-1}$. 
 \begin{enumerate}
 \item We have
\begin{equation}
 \lim_{j\rightarrow \infty} j^{\frac1{p}}\mu_j\big(Q^{-1}\big)= \bigg[ \frac1{n} (2\pi)^{-n} \int_{S^*M} \tr_E\big[ |\sigma(Q)(x,\xi)|^{-p} \big] dx d\xi \bigg]^{\frac1{p}}.
 \label{eq:Weyl.Weyl-|Q|}  
\end{equation}

\item If $Q$ is selfadjoint, then 
\begin{equation}
 \lim_{j\rightarrow \infty} j^{\frac1p} \lambda^\pm_j\big(Q^{-1}\big)= \bigg[ \frac1{n} (2\pi)^{-n} \int_{S^*M} \tr_E\big[ \sigma(Q)(x,\xi)_\pm^{-p} \big] dx d\xi \bigg]^{\frac1{p}}.
 \label{eq:Weyl.Weyl-Qpm}
\end{equation}
\end{enumerate}
\end{corollary}
\begin{proof}
 The first part is a mere restatement of the Weyl's law for $|Q|$. Namely, Proposition~\ref{prop:Weyl.NCR-zeta} for $A=1$ implies that the function 
 $\Tr[|Q|^{-s}]=\sum_{j\geq 0} \mu_j(|Q|^{-1})^{s}$ has a meromorphic extension to the half-space $\Re s> p-1/m$ with a single pole at $s=p$ such that
 \begin{equation*}
 \frac{1}{m}\Res\big(|Q|^{-p}\big)= \frac1{m} (2\pi)^{-n} \int_{S^*M} \tr_E\big[ |\sigma(Q)(x,\xi)|^{-p} \big] dx d\xi.
\end{equation*}
By using Ikehara's Tauberian theorem (see, e.g., \cite{Sh:Springer01}) we then obtain the Weyl's law~(\ref{eq:Weyl.Weyl-|Q|}). 

Suppose now that $Q$ is selfadjoint. Let $\Pi_0(Q)$ be the orthogonal projections onto $\ker Q$ and $\Pi_{\pm}(Q)$ the orthogonal projections onto the positive and negative eigenspaces of $Q$. Here $\Pi_0(Q)$ is a smoothing operator, and $\Pi_\pm(Q)$ are \psidos\ of order~$\leq 0$, since 
\begin{equation*}
 \Pi_\pm(Q)=\frac12\big(1-\Pi_0(Q)\pm Q|Q|^{-1}\big). 
\end{equation*}
In particular, $\sigma(\Pi_\pm(Q))= \frac12(1\pm \sigma(Q)\sigma(|Q|^{-1}))=\Pi_{\pm}(\sigma(Q))$. Therefore, Proposition~\ref{prop:Weyl.NCR-zeta}  for $A=\Pi_{\pm}(Q)$ shows that the function $\Tr[\Pi_{\pm}(Q)|Q|^{-s}]=\sum_{j\geq 0} \lambda_j^{\pm}(Q^{-1})^s$ has a meromorphic extension to the half-space $\Re s> p-1/m$ with a single pole at $s=p$ such that
 \begin{align*}
 \frac{1}{m}\Res\big(\Pi_\pm(Q)|Q|^{-p}\big) &= \frac1{m} (2\pi)^{-n} \int_{S^*M} \tr_E\big[ \Pi_{\pm}(\sigma(Q))|\sigma(Q)(x,\xi)|^{-p} \big] dx d\xi \\
 & = \frac1{m} (2\pi)^{-n} \int_{S^*M} \tr_E\big[ \sigma(Q)(x,\xi)_\pm^{-p} \big] dx d\xi .
\end{align*}
As above, using Ikehara's Tauberian theorem gives~(\ref{eq:Weyl.Weyl-Qpm}). The proof is complete. 
\end{proof}

We will also need the following special case of the Birman-Koplienko-Solomyak inequality.  

\begin{lemma}[{\cite[Theorem 3]{BKS:IVUZM75}}; see also~{\cite[Proposition~4.9]{BS:JSM92}}] \label{lem:Weyl.BKS}
Let $A$ and $B$ be non-negative selfadjoint operators on $\sH$ such that $A-B\in \sL_{q,\infty}$, $q>0$. Then 
$\sqrt{A}-\sqrt{B}\in \sL_{2q,\infty}$, and  
\begin{equation*}
 \big\| \sqrt{A}-\sqrt{B}\big\|_{2q,\infty} \leq C_q \sqrt{\|A-B\|_{q,\infty}},
\end{equation*}
 where the constant $C_q$ depends only on $q$. 
\end{lemma}

We are now in a position to prove Theorem~\ref{thm:Bir-Sol-asymp}. 

\begin{proof}[Proof of Theorem~\ref{thm:Bir-Sol-asymp}] 
 Let $P\in \Psi^m(M,E)$, $m>0$, and set $p=nm^{-1}$. Throughout this proof we let $\Delta_E=\nabla^*\nabla$ be the Laplacian of some Hermitian connection on $E$. In particular, $\sigma(\Delta_E)=|\xi|^2\op{id}_{E_x}$ and $\Delta_E^{-m}\in \Psi^{-2m}(M,E)$. Given $\epsilon>0$ set
 \begin{equation*}
 A_\epsilon =\sqrt{P^*P+\epsilon^2\Delta_E^{-m}}. 
\end{equation*}
Here $A_\epsilon^2-P^*P-=\epsilon^2\Delta_E^{-m}\in \sL_{p/2,\infty}$. Therefore, by Lemma~\ref{lem:Weyl.BKS} the difference $A_\epsilon-|P|$ is in $\sL_{p,\infty}$, and we have
\begin{equation*}
 \big\| A_\epsilon-|P|\big\|_{p,\infty} \leq C_p\epsilon \sqrt{\big\|\Delta_E^{-m}\big\|_{p/2,\infty}},
\end{equation*}
 where the constant $C_p$ does not depend on $\epsilon$. Thus,
\begin{equation}
 A_\epsilon \longrightarrow |P| \qquad \text{in $\sL_{p,\infty}$ as $\epsilon \rightarrow 0$}.
 \label{eq:Weyl.Aeps-|P|} 
\end{equation}

Let $Q_\epsilon\in \Psi^m(M,E)$ have principal symbol $\big( |\sigma(P)(x,\xi)|^2+\epsilon^2 |\xi|^{-2m}\big)^{-1/2}=\sigma(A_\epsilon^2)^{-1/2}$. In particular,  $Q_\epsilon$ is an elliptic operator. Thus, by Corollary~\ref{cor:Weyl.Weyl-invere-elliptic}  the inverse absolute value $|Q_\epsilon|^{-1}$ is a Weyl operator in $\sL_{p,\infty}$, and we have
\begin{align*}
 \Lambda\big(|Q_\epsilon|^{-1}\big) & = \bigg[ \frac1{n} (2\pi)^{-n} \int_{S^*M} \tr_E\big[\big( |\sigma(P)(x,\xi)|^2+\epsilon^2 |\xi|^{-2m}\big)^{\frac{p}{2}}\big] dxd\xi \bigg]^{\frac1{p}}\\
 & \xrightarrow[\epsilon \to 0]{~}    \bigg[  \frac1{n} (2\pi)^{-n}\int_{S^*M} \tr_E\big[ |\sigma(P)(x,\xi)|^p\big] dxd\xi \bigg]^{\frac1{p}}.
 \end{align*}

By construction $\sigma(|Q_\epsilon|^{-2})=\sigma(A_\epsilon^2)$, so $A_\epsilon^2-|Q_\epsilon|^{-2}$ is an operator in $\Psi^{-2m-1}(M,E)$, and hence is in the weak Schatten class $\sL_{q/2,\infty}$ with $q=2n(2m+1)^{-1}<p$. Lemma~\ref{lem:Weyl.BKS} then ensures us that $A_\epsilon -|Q_\epsilon|^{-1}$ is in the weak Schatten class $\sL_{q,\infty}$, and hence is contained in $(\sL_{p,\infty})_0$. It then follows from Corollary~\ref{cor:Bir-Sol.lunfo-perturbation-sing} that $A_\epsilon$ is a Weyl operator in $\sL_{p,\infty}$, and we have
\begin{equation*}
\Lambda(A_\epsilon)=  \Lambda\big(|Q_\epsilon|^{-1}\big)\xrightarrow[\epsilon \to 0]{~}  \bigg[ \frac1{n} (2\pi)^{-n} \int_{S^*M} \tr_E\big[ |\sigma(P)(x,\xi)|^p\big] dxd\xi \bigg]^{\frac1{p}}. 
 \end{equation*}
Combining this with  Corollary~\ref{cor:Bir-Sol.closedness-sing} and~(\ref{eq:Weyl.Aeps-|P|}) then shows that $|P|$ is a Weyl operator in $\sL_{p,\infty}$, and we have
\begin{equation*}
 \Lambda(|P|)= \lim_{\epsilon \rightarrow 0} \Lambda(A_\epsilon)= \bigg[ \frac1{n} (2\pi)^{-n} \int_{S^*M} \tr_E\big[ |\sigma(P)(x,\xi)|^p\big] dxd\xi \bigg]^{\frac1{p}}. 
\end{equation*}

Suppose now that $P^*=P$. Set $B_\epsilon =\frac12(A_\epsilon +P)$. It follows from~(\ref{eq:Weyl.Aeps-|P|}) that
\begin{equation}
 B_\epsilon \xrightarrow[\epsilon \to 0]{~} \frac12(|P|+P)=P_+ \qquad \text{in $\sL_{p,\infty}$}.
 \label{eq:Weyl.Beps-P+} 
\end{equation}
In addition, let $\tilde{Q}_\epsilon\in \Psi^m(M,E)$ be selfadjoint and have principal symbol
\begin{equation*}
 \left(\frac12 \sqrt{|\sigma(P)(x,\xi)|^2+\epsilon^2|\xi|^{-2m}}+\frac12\sigma(P)(x,\xi)\right)^{-1}= 
 \left( \frac12\sigma\big(|Q_\epsilon|^{-1}\big)+\frac12\sigma(P)(x,\xi)\right)^{-1}. 
\end{equation*}
As $\tilde{Q}_\epsilon$ is elliptic, Corollary~\ref{cor:Weyl.Weyl-invere-elliptic}  ensures that $\tilde{Q}_\epsilon^{-1}$ is a Weyl operator in $\sL_{p,\infty}$, and we have
\begin{equation*}
 \Lambda^+\big(\tilde{Q}^{-1}_\epsilon\big)  =\bigg[ \frac1{n} (2\pi)^{-n} \int_{S^*M}
  \tr_E\bigg[\bigg( \frac12 \sqrt{|\sigma(P)(x,\xi)|^2+\epsilon^2|\xi|^{-2m}}+\frac12 \sigma(P)(x,\xi)\bigg)^{p}\bigg] dxd\xi \bigg]^{\frac1{p}}. 
\end{equation*}
In particular, 
\begin{align*}
 \lim_{\epsilon \rightarrow 0}  \Lambda^+\big(\tilde{Q}^{-1}_\epsilon\big) & =\bigg[  \frac1{n} (2\pi)^{-n}\int_{S^*M}
  \tr_E\left[\left(\frac12 |\sigma(P)(x,\xi)|+\frac12\sigma(P)(x,\xi)\right)^{p}\right] dxd\xi \bigg]^{\frac1{p}}\\ 
& = \bigg[ \frac1{n} (2\pi)^{-n} \int_{S^*M} \tr_E\left[\sigma(P)(x,\xi)_+^{p}\right]dx d\xi \bigg]^{\frac1{p}}. 
\end{align*}

By construction, $\sigma(\tilde{Q}_\epsilon^{-1})=\frac{1}{2}(\sigma(|Q_\epsilon|^{-1}) +\sigma(P))$. Thus, $\tilde{Q}_\epsilon^{-1}-\frac{1}{2}(|Q_\epsilon|^{-1}+P)$ is a \psido\ of order~$\leq -(m+1)$, and hence is contained in $\sL_{\tilde{q},\infty}$ with $\tilde{q}=n(m+1)^{-1}$. As $\tilde{q}<p$, this implies that $\tilde{Q}_\epsilon^{-1}-\frac{1}{2}(|Q_\epsilon|^{-1}+P)$ is contained in $(\sL_{p,\infty})_0$. We know that $|Q_\epsilon|^{-1}-A_\epsilon$ is in $(\sL_{p,\infty})_0$ as well. Thus, $\tilde{Q}_\epsilon$ agrees with $\frac12(A_\epsilon+P)=B_\epsilon$ up to an operator in $(\sL_{p,\infty})_0$. It then follows from Corollary~\ref{cor:Bir-Sol.lunfo-perturbation}  that $B_\epsilon$ is a Weyl operator in $\sL_{p,\infty}$, and we have
\begin{equation*}
 \Lambda^+(B_\epsilon)=\Lambda^+\big(\tilde{Q}_\epsilon^{-1}\big)\xrightarrow[\epsilon \to 0]{~}  \bigg[  \frac1{n} (2\pi)^{-n}\int_{S^*M} \tr_E\left[\sigma(P)(x,\xi)_+^{p}\right]dx d\xi \bigg]^{\frac1{p}}. 
\end{equation*}
Combining this with~(\ref{eq:Weyl.Beps-P+}) and using Proposition~\ref{thm:Bir-Sol-asymp} shows that $P_+$ is a 
a Weyl operator in $\sL_{p,\infty}$, and we have
\begin{equation*}
\Lambda(P_+)= \lim_{\epsilon \to 0} \Lambda^+(B_\epsilon)= \bigg[  \frac1{n} (2\pi)^{-n}\int_{S^*M} \tr_E\left[\sigma(P)(x,\xi)_+^{p}\right]dx d\xi \bigg]^{\frac1{p}}. 
\end{equation*}
Upon replacing $P$ by $-P$ further shows that $P_{-}$ is a Weyl operator in $\sL_{p,\infty}$ and $\Lambda(P_-)$ is given by~(\ref{eq:Weyl.Bir-Sol-selfadjoint}). This shows that $P$ is a Weyl operator in $\sL_{p,\infty}$. The proof of Theorem~\ref{thm:Bir-Sol-asymp} is complete. 
\end{proof}

\begin{remark}
 A similar approach allows us to obtain a Weyl's law for negative order \psidos\ on NC tori (see~\cite{Po:Weyl-NC-tori}). 
\end{remark}

\begin{remark}
 As mentioned in Remark~\ref{rmk.Bir-Sol.non-smooth}, the  original version of Birman-Solomyak's result on $\R^n$ in~\cite{BS:VLU77} was established for  \psidos\ associated with anisotropic symbols. Therefore, we may expect to have a version of Birman-Solomyak's result for the Heisenberg calculus~\cite{BG:CHM, Ta:NCMA} and more generally for the pseudodifferential calculus on filtered manifolds~\cite{Me:Preprint82}. In those settings the pseudodifferential operators are defined in terms of anisotropic symbols. Note that we already have a noncommutative residue trace for the Heisenberg calculus (see~\cite{Po:JFA07}). 
\end{remark}

\section{Connes' Integration and Semiclassical Analysis} \label{sec:SC}
In this section, we look at the relationship between Connes' integration and semiclassical Weyl's laws for abstract Schr\"odinger operators, i.e., we shall answer Question~E. This will follow from the Birman-Schwinger principle. Once again our aim is to stress out how general and simple this relationship is (compare~\cite{MSZ:arXiv21}).

In what follows we let $H$ be a (densely defined) selfadjoint operator on $\sH$ with non-negative spectrum containing $0$. Its quadratic form $Q_H$ has domain $\dom(Q_H)=\dom (H+1)^{\frac12}$. We denote by $\sH_{+}$ the Hilbert space obtained by endowing $\dom(Q_H)$ with the Hilbert space norm, 
\begin{equation*}
 \|\xi\|_{+}=\big(Q_H(\xi,\xi)+\|\xi\|^2\big)^{\frac12}=\big\|(1+H)^{1/2}\xi\big\|, \qquad \xi \in \dom(Q_H). 
\end{equation*}
We also let $\sH_{-}$ be the Hilbert space of continuous \emph{anti-linear} functionals on $\sH_{+}$. Note that we have a continuous inclusion $\iota:\sH\hookrightarrow \sH_{-}$ given by 
\begin{equation*}
 \acou{\iota(\xi)}{\eta}=\scal{\xi}{\eta}, \qquad \xi\in \sH, \ \eta \in \sH_+. 
\end{equation*}
The operator $(H+1)$ is a unitary isomorphism from $\sH_+$ onto $\sH$ with inverse $(H+1)^{-1/2}:\sH\rightarrow \sH_{+}$. By duality we get a unitary isomorphism 
$(H+1)^{-1/2}:\sH_{-}\rightarrow \sH$ such that
\begin{equation*}
 \bigscal{(H+1)^{-1/2}\xi}{\eta}=\bigacou{\xi}{(H+1)^{-1/2}\eta}, \qquad \xi\in \sH_{-}, \ \eta \in \sH.
\end{equation*}

Let $V:\sH_{+}\rightarrow \sH_{-}$ be a bounded operator. We denote by $Q_V$ the corresponding quadratic form with domain $\sH_+$ and given by
\begin{equation*}
 Q_V(\xi,\eta):=\acou{V\xi}{\eta}, \qquad \xi,\eta\in \sH_+. 
\end{equation*}
We assume that $Q_V$ is \emph{symmetric} and \emph{$H$-form compact}. The latter condition means that the operator $V:\sH_{+}\rightarrow \sH_{-}$ is compact, or equivalently, $(H+1)^{-1/2}V(H+1)^{-1/2}$ is a compact operator on $\sH$. 

Our main focus is the operator $H_V:=H+V$. It makes sense as a bounded operator $H_V:\sH_+\rightarrow \sH_-$. As the symmetric quadratic form $Q_V$ is $H$-form compact, it is $H$-form bounded with zero $H$-bound (see~\cite[\S7.8]{Si:AMS15}). Therefore, by the KLMN theorem (see, e.g., \cite{RS2:1975, Sc:Springer12}) the restriction of $H_V$ to $\dom(H_V):=H_V^{-1}(\sH)$ is a bounded from below selfadjoint operator on $\sH$ whose quadratic form is precisely $Q_H+Q_V$. 

It can be further shown that, for all $\lambda \not\in \Sp(H)\cup \Sp(H_V)$ that $H$ and $H_V$ have the same essential spectrum (see, e.g., \cite[Theorem 7.8.4]{Si:AMS15}). Thus, as $H$ has non-negative spectrum, the bottom of the essential spectrum of $H_V$ is~$\geq 0$. 

As $H_{V}$ is bounded from below, we may list its eigenvalues below the essential spectrum as a non-decreasing sequence,  
\begin{equation*}
\lambda_0(H_V)\leq \lambda_1(H_V)\leq\lambda_2(H_V)\leq \cdots ,
\end{equation*}
where each eigenvalue is repeated according to multiplicity. This sequence may be finite or infinite, or even empty. We then introduce the counting function, 
\begin{align}
 N(H_V;\lambda) : = &\#\big\{j; \ \lambda_j(H_V)< \lambda\big\}, 
 \qquad \lambda <\inf  \Sp_{\text{ess}}(H_V). 
 \label{eq:CLR.counting}
\end{align}
We also set $N^{-}(H_V)=N(H_V;0)$. 

Assume further that  $0$ lies in the discrete spectrum of $H$, i.e., $0$ is an isolated eigenvalue of $H$. Thus, the essential spectrum of $H$ is contained in some  interval $[a, \infty)$ with $a>0$. We denote by $H^{-1/2}$ the partial inverse of $H^{1/2}$.  Moreover, as $H$ and $H_V$ have the same essential spectrum, it follows that $H_V$ has at most finitely many non-positive eigenvalues.

The Birman-Schwinger principle was established by Birman~\cite{Bi:AMST66} and Schwinger~\cite{Sc:PNAS61} for Schr\"odinger operators $\Delta+V$ on $\R^n$, $n\geq 3$. Its abstract version~\cite[Lemma~1.4]{BS:AMST89} (see also~\cite[Proposition~7.9]{MP:Part1}) allows us to relate the number of negative eigenvalues of $H_V$ to the counting functions of the Birman-Schwinger operator $H^{-\frac12}VH^{-1/2}$. Note also that the assumptions on $V$ ensure us that  $H^{-\frac12}VH^{-\frac12}$ is  a selfadjoint compact operator. 

\begin{proposition}[Abstract Birman-Schwinger Principle~\cite{BS:AMST89}; see also~\cite{MP:Part1}] Under the above assumptions, we have 
\begin{equation}
 N^{-}\left(H^{-\frac12}VH^{-\frac12};1\right)\leq N^{-}(H_{V}) \leq N^{-}\left(H^{-\frac12}VH^{-\frac12};1\right) + \dim \ker H. 
 \label{eq:SC.Birman-Schwinger}
\end{equation}
\end{proposition}

\begin{corollary}[Birman-Solomyak; see~{\cite[Theorem~10.1]{BS:TMMS72}} and~{\cite[Appendix~6]{BS:AMST80}}]\label{cor:SC.NHV=NKV} 
Assume further that $H^{-1/2}VH^{-1/2}$ is a Weyl operator in $\sL_{p,\infty}$ for some $p>0$. Then, under the semiclassical limit $h\rightarrow 0^+$, we have
\begin{equation}
 \lim_{h\rightarrow 0^+} h^{2p}N^{-}\big( h^2H+V\big) = \lim_{j\rightarrow \infty} j\lambda_j^{-}\left(H^{-\frac12}VH^{-\frac12}\right)^p.
 \label{eq:SC.NHV=NKV-SC} 
\end{equation}
\end{corollary}
\begin{proof}
The Birman-Schwinger principle~(\ref{eq:SC.Birman-Schwinger}) ensures that, as $h\rightarrow 0^+$, we have
\begin{align}
 N^{-}\big( h^2H+V\big) &= N^{-}(H+h^{-2}V), \nonumber\\
 &=  N^{-}\left(H^{-\frac12}h^{-2}VH^{-\frac12};1\right) +\op{O}(1),
 \label{eq:SC.NHV-NKV-SC}\\
 & = N^{-}\left(H^{-\frac12}VH^{-\frac12};h^2\right) +\op{O}(1). \nonumber
\end{align}
If  $H^{-\frac12}VH^{-\frac12}$ is a Weyl operator in $\sL_{p,\infty}$ for some $p>0$, then by Remark~\ref{rmk:Bir-Sol.counting} we have
\begin{equation*}
 \lim_{h \rightarrow 0^+} h^{2p}N^{-}\left(H^{-\frac12}VH^{-\frac12};h^2\right)= \lim_{j\rightarrow \infty} j\lambda_j^-\left(H^{-\frac12}VH^{-\frac12}\right)^p. 
\end{equation*}
Combining this with~(\ref{eq:SC.NHV-NKV-SC}) immediately gives the result. 
\end{proof}

Combining the above corollary with Proposition~\ref{prop:Bir-Sol.Weyl-mesurable}  leads us to the following semiclassical interpretation of Connes' integral.

\begin{proposition}\label{prop:SC.NHV-bint}
 Assume that $V\geq 0$ and $H^{-\frac12}VH^{-\frac12}$ is a Weyl operator in $\sL_{1,\infty}$. Then, the operator $H^{-\frac12}VH^{-\frac12}$ is strongly measurable, and we have 
\begin{equation}
  \lim_{h \rightarrow 0^+} h^2 N^{-}\big( h^2H-V\big)=  \bint H^{-\frac12}VH^{-\frac12}.
  \label{eq:SC.NHV-bint}   
\end{equation}
 \end{proposition}

 In a recent preprint McDonald-Sukochev-Zanin~\cite{MSZ:arXiv21} obtained a related formula in the setting of spectral triples under some additional 
 technical conditions. We stress out that, on the one hand, the equality~(\ref{eq:SC.NHV-bint}) holds in fairly great generality. For instance, we don't have to require $H$ to have compact resolvent. On the other hand, this equality is an immediate consequence of the abstract Birman-Schwinger principle (compare~\cite{MSZ:arXiv21}). 

In any case, the formula~(\ref{eq:SC.NHV-bint}) highlights a neat link between the semiclassical analysis of Schr\"odinger operators and Connes' noncommutative geometry. These are two different approaches to quantum theory, and so it is quite interesting to witness some intersection between them. 

To illustrate the above results, suppose that $(M^n,g)$ be a closed Riemannian manifold. Let $\Delta_g$ be the corresponding Laplace-Beltami operator acting on functions, and take $H=\Delta_g^{n/2}$. In this setup $\sH_{+}$ is the Sobolev space $W^{n/2}_2(M)$ and $\sH_{-}$ is the antilinear dual $W^{-n/2}_2(M)$. Let $V(x)$ be a real-valued potential in the Orlicz class $\LlogL(M)$, i.e., a measurable function such that
\begin{equation*}
 \int (1+|V(x)|)\log(1+|V(x)|)\sqrt{g(x)}dx<\infty.  
\end{equation*}
For instance, we may take $V(x)$ to be in $L_p(M)$ for any $p>1$.

It can be shown that $V(x)$ gives rise to a bounded operator $V:W^{n/2}(M)\rightarrow W^{-n/2}(M)$ (see~\cite{Ro:arXiv21}), and so $\Delta_g^{-n/4}V\Delta_g^{-n/4}$ is bounded on $L_2(M)$. As mentioned in Remark~\ref{rmk:Strong.Orlicz}, results of Rozenblum~\cite{Ro:arXiv21} and Sukochev-Zanin~\cite{SZ:arXiv21} (see also~\cite{Po:Weyl-Orlicz, RS:EMS21}) ensure us that $\Delta_g^{-n/4}V\Delta_g^{-n/4}$ is a Weyl operator in $\sL_{1,\infty}$. Thus, $\Delta_g^{n/2}+V$ makes sense as a form sum as above, and by Corollary~\ref{cor:SC.NHV=NKV} we have
\begin{align*}
  \lim_{h\rightarrow 0^+} h^{n}N^{-}\big( h^n\Delta_g^{n/2}+V\big)& = \lim_{j\rightarrow \infty} j\lambda_j^{-}\left(\Delta_g^{-\frac{n}{4}}V\Delta_g^{-\frac{n}{4}}\right)\\
  & = \frac1{n} (2\pi)^{-n}  \int_M  V(x)_{-}\sqrt{g(x)}dx. 
\end{align*}
If $V(x)\geq 0$, then by using Proposition~\ref{prop:SC.NHV-bint} we further obtain
\begin{align*}
 \bint \Delta_g^{-\frac{n}{4}}V\Delta_g^{-\frac{n}{4}} & = \lim_{h \rightarrow 0^+} h^nN^{-}\big(h^n \Delta_g^{\frac{n}{2}}-V\big),\\
 &=  \frac1{n} (2\pi)^{-n}  \int_M  V(x)\sqrt{g(x)}dx. 
 \end{align*}
This provides us with a semiclassical interpretation of Connes' integration formula~(\ref{eq:NCInt.Connes-int-formula}). (See also~\cite{Po:Weyl-Orlicz} for an extension of this result to matrix-valued Orlicz-$\LlogL$ potentials.) 

\appendix

\section{Embeddings of Hilbert spaces} \label{sec:Appendix}
In this section, we gather a few basic facts about embedding of Hilbert spaces and their actions on Schatten and weak Schatten classes. 

Given quasi-Banach spaces $\sE$ and $\sE'$, a continuous linear embedding $\iota:\sE'\rightarrow \sE$ is a continuous linear map which is one-to-one and has closed range. For instance, any isometric linear map $\iota:\sE'\rightarrow \sE$ is an embedding.

Suppose now that $\sH$ and $\sH'$ are Hilbert spaces, and let $\iota: \sH'\rightarrow \sH$ be a continuous linear embedding. Denote by $\sH_1$ the range of $\iota$. By assumption this is a closed subspace of $\sH$. The embedding $\iota:\sH'\rightarrow \sH$ induces a continuous linear isomorphism  $\iota:\sH'\rightarrow \sH_1$ whose inverse is denoted $\iota^{-1}:\sH_1\rightarrow \sH'$. Let $\pi:\sH\rightarrow \sH$ be the orthogonal projection onto $\sH_1$. Then $\iota^{-1}\circ \pi$ is a left-inverse of $\iota$. More precisely, 
\begin{equation*}
 (\iota^{-1}\circ \pi)\circ \iota =\op{id}_{\sH'}, \qquad  \iota\circ (\iota^{-1}\circ \pi)=\pi. 
\end{equation*}
The pushforward $\iota_*:\sL(\sH')\rightarrow \sL(\sH)$ is then defined by
\begin{equation}
 \iota_*A= \iota \circ A \circ (\iota^{-1}\circ \pi), \qquad A \in \sL(\sH'). 
 \label{eq:App.embedding-sLsH}
\end{equation}
In fact, with respect to the orthogonal splitting $\sH=\sH_1\oplus \sH_1^\perp$ we have 
\begin{equation}\label{eq:App.iota*A}
 \iota_*A= 
\begin{pmatrix}
 \iota A \iota^{-1} & 0 \\
 0 & 0
\end{pmatrix}. 
\end{equation}
In particular, we see that the pushforward map $\iota_*:\sL(\sH')\rightarrow \sL(\sH)$ is a continuous embedding. It is also multiplicative, and so we get an embedding of (unital) Banach algebras.  

If in addition $\iota$ is an isometric embedding, then $\iota_*:\sL(\sH')\rightarrow \sL(\sH)$ is an isometric embedding as well. In this case, we $\iota:\sH'\rightarrow \sH_1$ is a unitary operator, and so in view of~(\ref{eq:App.iota*A}) we have $\iota_*(A^*)=(\iota_*A)^*$. Thus, in this case $\iota_*:\sL(\sH')\rightarrow \sL(\sH)$ even is an (isometric) embedding of $C^*$-algebras. 

In any case the pushforward  $\iota_*:\sL(\sH')\rightarrow \sL(\sH)$ maps compact operators on $\sH'$ to compact operators on $\sH$. Moreover, in view of~(\ref{eq:App.iota*A}) we have the following result. 

\begin{proposition}\label{prop:App.eigenvalues} 
 If $A$ is a compact operator on $\sH'$, then $A$ and $\iota_*A$ have the same non-zero eigenvalues with the same algebraic multiplicities. In particular, any eigenvalue sequence for $A$ is an eigenvalue sequence for $\iota_*A$, and vice versa.  
\end{proposition}

Suppose that $\iota: \sH'\rightarrow \sH$ is an isometric embedding. As mentioned above $\iota_*:\sL(\sH')\rightarrow \sL(\sH)$ is an isometric embdedding of $C^*$-algebras. Thus, for any $A \in \sL(\sH')$ and $f\in C(\Sp(A))$ we have 
\begin{equation*}
 f(\iota_*A)= \begin{pmatrix}
 \iota f(A) \iota^{-1} & 0 \\
 0 & 0
\end{pmatrix} =\iota_*f(A). 
\end{equation*}
In particular, we have $|\iota_*A|=\iota_*|A|$, and so $|\iota_*A|$ and $\iota_*|A|$ have the same non-zero eigenvalues with same multiplicity. This means that $A$ and $\iota_*A$ have the same singular value sequence. 

In general, as $\iota:\sH'\rightarrow \sH_1$ is a continuous linear isomorphism, we can pullback the inner product on $\sH_1$ to a new inner product on $\sH'$, which is equivalent to its original inner product.
With respect to this new inner product the embedding $\iota:\sH'\rightarrow \sH$ becomes isometric. Thus, given any compact operator $A$ on $\sH'$, the singular value sequence of $\iota_*A$ agrees with the singular value sequence of $A$ with respect to the new inner product on $\sH'$. As this inner product is equivalent to the original inner product of $\sH'$, by using the min-max principle~(\ref{eq:min-max}) we eventually arrive at the following result.

\begin{proposition}\label{prop:App.embed-Schatten}
 Let  $\iota:\sH'\rightarrow \sH$ be a continuous linear embedding. 
 \begin{enumerate}
 \item There is $c>0$ such that, for every compact operator $A$ on $\sH'$, we have 
 \begin{equation*}
 c^{-1} \mu_j(A)\leq \mu_j\big(\iota_*A) \leq c \mu_j(A) \qquad \forall j\geq 0. 
\end{equation*}
We may take $c=1$ when $\iota$ is an isometric embedding. 

\item Given any $p>0$, the operator $A$ is in the class $\sL_{p}(\sH')$ (resp., $\sL_{p,\infty}(\sH')$)  if and only if $\iota_*A$ is in 
$\sL_{p}(\sH)$ (resp., $\sL_{p,\infty}(\sH)$). Moreover, the pushforward map~(\ref{eq:App.embedding-sLsH}) induces continuous linear embeddings, 
 \begin{equation*}
 \iota_*: \sL_p(\sH')\longrightarrow \sL_p(\sH), \qquad \iota_*: \sL_{p,\infty}(\sH')\longrightarrow \sL_{p,\infty}(\sH). 
\end{equation*}
These embeddings are isometric whenever $\iota$ is an isometric embedding.
\end{enumerate}
 \end{proposition}

\begin{remark}
 Part~(ii) holds more generally for any symmetrically normed ideal. 
\end{remark}

\subsection*{Aknowledgements} I wish to thank Alain Connes, Grigori Rozenblum, Edward McDonald, and Fedor Sukochev for various stimulating discussions related to the subject matter of this article. I also thank the University of Ottawa for its hospitality during the whole preparation of this paper.

\end{document}